\documentclass[preprint,12pt]{elsarticle}

\usepackage{amssymb}
\usepackage{amsmath}
\usepackage{amsthm}
\usepackage{geometry}
\usepackage{graphicx}
\usepackage{hyperref}
\usepackage{algorithm}
\usepackage{algorithmic}
\allowdisplaybreaks

\newtheorem{theorem}{Theorem}
\newtheorem{lemma}[theorem]{Lemma}
\newtheorem{remark}[theorem]{Remark}
\newtheorem{definition}[theorem]{Definition}

\newtheorem{proposition}[theorem]{Proposition}

\makeatletter
\newenvironment{breakablealgorithm}
{
    \begin{center}
        \refstepcounter{algorithm}
        \hrule height.8pt depth0pt \kern2pt
        \renewcommand{\caption}[2][\relax]{
            {\raggedright\textbf{\ALG@name~\thealgorithm} ##2\par}%
            \ifx\relax##1\relax 
            \addcontentsline{loa}{algorithm}{\protect\numberline{\thealgorithm}##2}%
            \else 
            \addcontentsline{loa}{algorithm}{\protect\numberline{\thealgorithm}##1}%
            \fi
            \kern2pt\hrule\kern2pt
        }
    }{
        \kern2pt\hrule\relax
    \end{center}
}
\makeatother

\journal{Journal of Differential Equations}

\begin{document}

\begin{frontmatter}



\title{An Efficient Symbolic Algorithm for Computing Lyapunov Constants in Planar Switching Systems}


\author[inst1]{Cheng Zheng}

\author[inst1]{Laigang Guo}

\author[inst2]{Xiangyu Wang}


\affiliation[inst1]{organization={School of Mathematical Sciences, Beijing Normal University},
            city={Beijing},
            postcode={100875},
            country={China}}

\affiliation[inst2]{organization={College of Mathematics and Physics, Beijing University of Chemical Technology},
            city={Beijing},
            postcode={100029},
            country={China}}

\nonumnote{%
\rightskip=0pt plus 1fil\relax
Corresponding authors: Xiangyu Wang (xwan387@buct.edu.cn),
Laigang Guo (lgguo@bnu.edu.cn).%
}


\begin{abstract}
Lyapunov constants are a crucial tool for analyzing bifurcation behavior in switching systems. This paper proposes an efficient algebraic symbolic algorithm for computing Lyapunov constants in planar switching systems. By mapping the vector fields into the complex domain, we construct a normal form algorithm based on Laurent polynomials. The classical Poincaré map method requires repeated symbolic integration, which may lead to substantial computational difficulties at high orders. Our approach improves computational efficiency by replacing continuous integration with direct algebraic operations. We apply this approach to investigate two switching models. For a generalized switching quartic Liénard system, the algebraic method extracts the center conditions and proves the existence of 10 small-amplitude limit cycles, establishing a new lower bound for its cyclicity. Furthermore, a physically motivated circuit-level modification of the Alpazur oscillator is investigated. We show that at most one small-amplitude limit cycle can bifurcate from the center and that this upper bound is attainable, illustrating the applicability of the proposed algorithm to a practical switching circuit. This application shows that our method provides a systematic framework for analyzing and interpreting complex dynamical phenomena arising in practical piecewise smooth systems.
\end{abstract}

\begin{keyword}


Symbolic algorithm, Planar switching system, Lyapunov constant, Limit cycle, Normal form
\MSC[2020] 34C07 \sep 34C20 \sep 68W30
\end{keyword}

\end{frontmatter}



\section{Introduction}
Piecewise smooth differential systems provide a natural framework for modeling abrupt changes in system dynamics \cite{filippov1988differential}. This formulation applies to physical phenomena from dry friction in mechanical oscillators to power regulation in ele 
 ctronic circuits \cite{kousaka1999bifurcation, bernardo2008piecewise, colombo2009two}. Periodic oscillations in such models motivate the study of limit cycles and their bifurcations. Planar switching systems may exhibit higher Hopf cyclicity than smooth systems of comparable polynomial degree. The corresponding center-focus problem is therefore an important step in determining local cyclicity. \cite{guckenheimer2013}. The local emergence of these small-amplitude limit cycles is governed by the Hopf bifurcation \cite{guckenheimer2013, marsden2012hopf}, and determining their precise number relies on computing the Lyapunov constants of the system \cite{perko2013differential, han2012normal}. 

Various methods have been developed to compute the Lyapunov constants for planar switching systems. Freire et al. and Llibre et al. \cite{freire1998bifurcation, llibre2004existence, buzzi2013piecewise} formulated generalized Poincaré maps by extending classical return-map techniques. Gasull and Torregrosa \cite{gasull2003center} developed a displacement-function approach based on the matching of half-return maps. For systems with an unperturbed center or a Hamiltonian structure, Han and Liang \cite{han2012normal, liang2016number} expanded higher-order Melnikov functions. This approach determines cyclicity by integrating piecewise perturbation terms along unperturbed closed orbits. Related study \cite{romanovski2009center} extended formal series expansions and singular-point quantities to nonsmooth systems. Tian and Yu \cite{tian2015center} applied algebraic expansions to resolve center conditions in a quadratic switching Bautin system and proved the existence of 10 limit cycles. Extending Lyapunov computations to cubic switching systems, Guo et al. \cite{guo2019bifurcation} investigated $Z_2$-equivariant cubic switching systems and constructed a specific example to prove the existence of 18 small-amplitude limit cycles bifurcating from two symmetric centers, and Yu et al. \cite{yu2021eighteen} obtained 18 limit cycles by perturbing two symmetric fine foci. 

Liénard systems form an important class in the study of limit cycles in both smooth and switching differential systems. Formulated as $\dot{x} = y - F(x), \dot{y} = -g(x)$, the generalized model reduces to the classical Liénard system when $g(x)=x$. Smale \cite{smale1998mathematical} posed the determination of the maximum number of limit cycles in this classical configuration as his 13th mathematical problem. Lins et al. \cite{lins1977lienard} conjectured that this cyclicity is bounded by $\lfloor (n-1)/2 \rfloor$, where $n$ is the polynomial degree of $F(x)$. Subsequent investigations \cite{tian2011hopf, christopher1999small} on generalized smooth Liénard equations established specific Hopf cyclicity bounds contingent upon the degree and symmetry of $g(x)$. Recent studies \cite{llibre2015limit, chen2024nilpotent} have extended this setting to piecewise smooth Liénard systems by allowing $F$ and $g$ to be piecewise polynomial. In particular,  through perturbation analysis, Chen and Peng \cite{chen2026global} characterized generalized piecewise cubic Liénard systems, isolated explicit global center conditions, and proved the existence of nine limit cycles. 

In this paper, we propose an algebraic symbolic algorithm for planar switching systems. The main contributions are as follows:
\begin{enumerate}
    \item We formulate a normal-form procedure in an extended Laurent algebra obtained from the substitution $z=e^{\mathrm{i}\theta}$. The resulting operators compute half-period means and solve homological equations by finite algebraic operations.
    \item We apply the algorithm to a generalized quartic switching Liénard system, derive several center conditions and obtain a parameter configuration associated with 10 small-amplitude limit cycles, establishing a new lower bound for such systems.
    \item We apply the method to a physically motivated modification of the Alpazur oscillator, derive its center conditions, and establish a sharp upper bound of one for the number of bifurcating small-amplitude limit cycles.
\end{enumerate}

The remainder of the paper is organized as follows. Section \ref{Preliminaries} reviews the necessary definitions and results for planar switching systems. Section \ref{Main Results} develops the extended Laurent-algebra framework and presents the recursive algorithm. Section \ref{Examples} applies the algorithm to two examples. Section \ref{Conclusion} concludes the paper.

\section{Preliminaries}\label{Preliminaries}

In this section, we review the preliminaries for planar switching systems with Hopf bifurcations. 

Consider a planar switching system separated by a boundary $\Sigma = \{\mathbf{X} \in \mathbb{R}^2 \mid H(\mathbf{X}) = 0\}$, where $\mathbf{X}=(x,y)^T$, $H: \mathbb{R}^2 \to \mathbb{R}$ is a scalar switching function that divides the state space into two regions, $\Sigma^+$ and $\Sigma^-$. To ensure the boundary $\Sigma$ is a well-defined one-dimensional smooth manifold passing through the equilibrium of interest, we assume that $H(\mathbf{X})$ is sufficiently smooth, $H(\mathbf{0}) = 0$, and the non-degeneracy condition $\nabla H(\mathbf{X}) \neq \mathbf{0}$ holds for all $\mathbf{X} \in \Sigma$. The system is described by 
\begin{equation}\label{sys:2.41}
    \dot{\mathbf{X}} = 
    \begin{cases} 
        A_1 \mathbf{X} + \mathbf{F}_1(\mathbf{X}), & H(\mathbf{X}) \ge 0, \\ 
        A_2 \mathbf{X} + \mathbf{F}_2(\mathbf{X}), & H(\mathbf{X}) < 0,
    \end{cases}
\end{equation}
where $A_k$ ($k=1, 2$) are $2 \times 2$ matrices with a pair of purely imaginary eigenvalues $\pm i \omega_k$ ($k=1, 2$), and $\mathbf{F}_1(\mathbf{X})$ and $\mathbf{F}_2(\mathbf{X})$ are the higher-order nonlinear terms with $\mathbf{F}_i(\mathbf{0})=0$, $i=1,2$.

\begin{lemma}\label{rotation transformation}
    Let the switching boundary be the straight line
    \begin{equation}\label{straight line}
        \Sigma = \{ \mathbf{X} \in \mathbb{R}^2 \mid a x + b y = 0 \},
    \end{equation}
    where $a^2 + b^2 \neq 0$. There exists a nonsingular linear transformation $\tilde{\mathbf{X}} = T\mathbf{X}$ that maps the boundary $\Sigma$ to the horizontal axis, i.e., $\tilde{y} = 0$. 
\end{lemma}

\begin{proof}
    Let
    \begin{equation*}
        T = \begin{pmatrix} b & -a \\ a & b \end{pmatrix}.
    \end{equation*}
    Since $\det(T) = a^2 + b^2 \neq 0$, $T$ is nonsingular and thus invertible. 
    Let $\tilde{\mathbf{X}} = (\tilde{x}, \tilde{y})^T = T\mathbf{X}$. Then we obtain
    \begin{equation*}
        \tilde{y} = a x + b y.
    \end{equation*}
    For every $\mathbf{X} \in \Sigma$, \eqref{straight line} gives $\tilde{y} = 0$, which completes the proof.
\end{proof}

\begin{lemma} \label{lem:canonical_form}
Consider the planar switching system \eqref{sys:2.41} with a straight-line boundary $\Sigma$. There exists a coordinate and time transformation that maps the system into the form 
\begin{align}\label{sys:2.5}
    \dot{{\mathbf{X}}} &= 
    \begin{cases} 
        J\mathbf{X} + \mathbf{F}^{(1)}(x,y), & \text{if } y \ge 0 , \\ 
        J\mathbf{X} + \mathbf{F}^{(2)}(x,y), & \text{if } y < 0 , \\
    \end{cases}
\end{align}
where $\mathbf{X}=(x,y)^T$, $J = \begin{pmatrix} 0 & -1 \\ 1 & 0 \end{pmatrix}$, and the functions $\mathbf{F}^{(i)}=(\sum\limits_{k=2}^n X_k^{(i)}, \sum\limits_{k=2}^n Y_k^{(i)})^T$, $i=1,2$ are analytic, starting from quadratic terms.
\end{lemma}

\begin{proof}
    By Lemma~\ref{rotation transformation}, we first apply the linear transformation $\mathbf{u} = T\mathbf{X}$ to align the switching boundary $\Sigma$ with the horizontal axis $v=0$ in the intermediate coordinates $\mathbf{u} = (u, v)^T$. The linear parts of the two subsystems are then governed by the matrices $\bar{A}_k = T A_k T^{-1}$ ($k=1, 2$), which retain the purely imaginary eigenvalues $\pm i\omega_k$.

    For each subsystem, introducing the piecewise time scaling $\tau = \omega_k t$ yields the rescaled linear matrices $\tilde{A}_k = \frac{1}{\omega_k} \bar{A}_k$, which now have eigenvalues $\pm i$. Thus, we can write
    \begin{equation*}
        \tilde{A}_k = \begin{pmatrix} a_k & b_k \\ c_k & -a_k \end{pmatrix},
    \end{equation*}
    where $a_k^2 + b_k c_k = -1$ and $c_k \neq 0$. To transform each $\tilde{A}_k$ into $J$, we apply the piecewise coordinate transformation $\mathbf{u} = P_k\mathbf{X}$ with
    \begin{equation*}
        P_k = \begin{pmatrix} 1 & a_k \\ 0 & c_k \end{pmatrix}.
    \end{equation*}
    A direct calculation gives $P_k^{-1}\widetilde A_kP_k=J$. Furthermore, on the switching boundary $y=0$, both transformations yield $u = x$ and $v = 0$. This ensures that the global transformation is continuous across the boundary, mapping $v=0$ precisely to $y=0$ and completing the proof.
\end{proof}

The preceding two lemmas establish that planar switching system \eqref{sys:2.41} can be reduced to the unified form \eqref{sys:2.5}. 

Consider the simplified planar switching system \eqref{sys:2.5}. For the subsystem defined in the lower half-plane ($y < 0$), we introduce the transformation $(x, y, t) \mapsto (x, -y, -t)$. Consequently, the original switching system is transformed into two subsystems, $S^{(1)}$ and $S^{(2)}$, both now defined on the common upper half-plane $\mathcal{D} = \{(x,y) \mid y \ge 0\}$. Assume $S^{(i)}$, $i=1,2$ have the form
\begin{equation}\label{poly}
    S^{(i)}: \begin{cases}
                \dot{x} = -y + P^{(i)}(x, y), \\
                \dot{y} = x + Q^{(i)}(x, y),
             \end{cases} 
\end{equation}
where
\begin{align}
    P^{(i)}(x, y) = \sum_{k=2}^{\infty} P_k^{(i)}(x, y), \quad Q^{(i)}(x, y) = \sum_{k=2}^{\infty} Q_k^{(i)}(x, y),
\end{align}
and $P_k^{(i)}$ and $Q_k^{(i)}$ are homogeneous polynomials of degree $k$ in variables $x$ and $y$.

Applying the standard polar coordinate transformation $x = r \cos \theta$ and $y = r \sin \theta$ to \eqref{poly}, we obtain
\begin{align*}
    \dot{r} &= \frac{1}{r} \sum_{k=2}^{\infty} \left[ xP_k^{(i)}(x, y) + yQ_k^{(i)}(x, y) \right] \nonumber \\
            &= \sum_{k=2}^{\infty} A_k^{(i)}(\theta) r^k,  \\
    \dot{\theta} &= 1 + \frac{\sum_{k=2}^{\infty} \left[ xQ_k^{(i)}(x, y) - yP_k^{(i)}(x, y) \right]}{r^2} \nonumber \\
                &= 1 + \sum_{k=2}^{\infty} B_k^{(i)}(\theta) r^{k-1}, 
\end{align*}
where
\begin{align*}
    A_k^{(i)}(\theta)&=\cos\theta P_k^{(i)}(\cos\theta,\sin\theta)+\sin\theta Q_k^{(i)}(\cos\theta,\sin\theta),\\
    B_k^{(i)}(\theta)&=\cos\theta Q_k^{(i)}(\cos\theta,\sin\theta)-\sin\theta P_k^{(i)}(\cos\theta,\sin\theta).
\end{align*}

Therefore, 
\begin{equation}
    \frac{dr}{d\theta} = \frac{\sum_{k=2}^{\infty} A_k^{(i)}(\theta)r^k}{1 + \sum_{k=2}^{\infty} B_k^{(i)}(\theta)r^{k-1}}.
\end{equation}
Expanding the right-hand side at $r=0$ gives
\begin{equation}\label{eq:dr/dtheta}
    \frac{dr}{d\theta}=\sum_{k=2}^{\infty} R_k^{(i)}(\theta)r^k,
\end{equation}
where
\begin{align*}
    R_2^{(i)}(\theta) &= A_2^{(i)}(\theta), \\
    R_3^{(i)}(\theta) &= A_3^{(i)}(\theta) - A_2^{(i)}(\theta)B_2^{(i)}(\theta), \\
    R_4^{(i)}(\theta) &= A_4^{(i)}(\theta) - A_2^{(i)}(\theta)B_3^{(i)}(\theta) - A_3^{(i)}(\theta)B_2^{(i)}(\theta) \\
    &\quad + A_2^{(i)}(\theta)(B_2^{(i)}(\theta))^2, \\
    &\dots
\end{align*}
Let $r_i(\theta,h)$ denote the solution of $S^{(i)}$ with $r_i(0,h)=h$, and write
\begin{align}\label{eq:ri}
    r_i &= \tilde{r}_i(\theta, h) = \sum_{k=1}^{\infty} u^{(i)}_k(\theta)h^k, 
\end{align}
satisfying the initial conditions $u^{(i)}_1(0) = 1$ and $u^{(i)}_k(0) = 0$ for any $k \ge 2$. Substituting \eqref{eq:ri} back into \eqref{eq:dr/dtheta} allows us to solve for the coefficients $u^{(i)}_k(\theta)$ iteratively. This yields
\begin{equation*}\label{substitute}
    \sum_{k=1}^{\infty} \frac{d u_k^{(i)}(\theta)}{d\theta}h^k = \sum_{m=2}^{\infty} R_m^{(i)}(\theta) \left( \sum_{k=1}^{\infty} u^{(i)}_k(\theta)h^k \right)^m.
\end{equation*}

By expanding the right-hand side of \eqref{substitute} and equating the coefficients of $h^k$ on both sides, we obtain a sequence of differential equations for $u^{(i)}_k(\theta)$
\begin{align*}
    \frac{d u_1^{(i)}(\theta)}{d\theta} &= 0, \nonumber\\
    \frac{d u_2^{(i)}(\theta)}{d\theta} &= R_2^{(i)}(\theta)\left(u_1^{(i)}(\theta)\right)^2, \nonumber\\
    \frac{d u_3^{(i)}(\theta)}{d\theta} &= 2R_2^{(i)}(\theta)u_1^{(i)}(\theta)u_2^{(i)}(\theta)+R_3^{(i)}(\theta)\left(u_1^{(i)}(\theta)\right)^3, \nonumber\\
    \frac{d u_4^{(i)}(\theta)}{d\theta} &= R_2^{(i)}(\theta)\left[\left(u_2^{(i)}(\theta)\right)^2+2u_1^{(i)}(\theta)u_3^{(i)}(\theta)\right] \nonumber\\
    &\quad +3R_3^{(i)}(\theta)\left(u_1^{(i)}(\theta)\right)^2u_2^{(i)}(\theta)+R_4^{(i)}(\theta)\left(u_1^{(i)}(\theta)\right)^4, \nonumber\\
    &\cdots
\end{align*}

With the initial conditions $u^{(i)}_1(0) = 1$ and $u^{(i)}_k(0) = 0$ for $k \ge 2$, these equations can be solved recursively
\begin{align*}
    u^{(i)}_1(\theta) &= 1, \nonumber \\
    u^{(i)}_2(\theta) &= \int_0^\theta R_2^{(i)}(\varphi) d\varphi, \nonumber \\
    u^{(i)}_3(\theta) &= \int_0^\theta \left[ 2R_2^{(i)}(\varphi) u_2(\varphi) + R_3^{(i)}(\varphi) \right] d\varphi, \nonumber \\
    u^{(i)}_4(\theta) &= \int_0^\theta \left[ R_2^{(i)}(\varphi) \left( u_2^2(\varphi) + 2u_3(\varphi) \right) + 3R_3^{(i)}(\varphi) u_2(\varphi) + R_4^{(i)}(\varphi) \right] d\varphi, \nonumber \\
    &\dots
\end{align*}

Under the transformation $(x,y,t)\mapsto(x,-y,-t)$, the negative half-return map of the lower subsystem becomes the positive half-return map of the folded subsystem $S^{(2)}$. We therefore define the successive functions for the two subsystems at $\theta = \pi$ as
\begin{align}
    \Delta_1(h) &= \tilde{r}_1(\pi, h) - h, \\
    \Delta_2(h) &= \tilde{r}_2(\pi, h) - h,
\end{align}
respectively. 

\begin{definition}\cite{gasull2003center}\label{def:Vk_poincare}
    Let
\begin{equation}
\begin{aligned}
    \Delta(h) &= \Delta_1(h) - \Delta_2(h) = \tilde{r}_1(\pi, h) - \tilde{r}_2(\pi, h) \\
    &= \sum_{k=1}^{\infty} \left( u^{(1)}_k(\pi) - u^{(2)}_k(\pi) \right) h^k \\
    &= \sum_{k=1}^{\infty} V_k h^k.
\end{aligned}
\end{equation}
Then $\Delta(h)$ is called the successive function, and $V_k$ is called the $k$th-order Lyapunov constant of the switching system \eqref{sys:2.5}.
\end{definition}

\begin{remark}
$V_1$ is determined by the linear part of the two subsystems. For system~\eqref{sys:2.5}, 
\begin{equation}
    V_1=u^{(1)}_1(\pi) - u^{(2)}_1(\pi)=0.
\end{equation}
\end{remark}

\begin{lemma}\label{lemma:normal_form}\cite{han2012normal}
There exists a near-identity transformation
\begin{equation}\label{eq:near_identity}
    r=\rho+\sum_{k=2}^{N}h_k^{(i)}(\theta)\rho^k,\quad h_k^{(i)}(0)=h_k^{(i)}(\pi)=0, \quad i=
    1,2,\quad N\ge2,
\end{equation}
which transforms \eqref{eq:dr/dtheta} into the following form
\begin{equation}\label{eq:normal_form}
    \frac{d\rho}{d\theta}=\sum_{k=2}^{N}g_k^{(i)}\rho^k+O(\rho^{N+1}).
\end{equation}
\end{lemma}

\begin{remark}
    The form given by \eqref{eq:normal_form} is referred to as the normal form of \eqref{eq:dr/dtheta}.
\end{remark}

\begin{proposition}\label{prop:normal_form_constant}
Assume that 
\begin{equation*}
    V_1=V_2=\cdots=V_{k-1}=0, \quad k\ge2.
\end{equation*}
Then 
\begin{equation}\label{eq:Vk_normal_form}
    V_k=\pi(g_k^{(1)}-g_k^{(2)}).
\end{equation}
\end{proposition}

\begin{proof}
Let $\rho_i(\theta;h)$ be the solution of \eqref{eq:normal_form} satisfying $\rho_i(0;h)=h$. Integrating \eqref{eq:normal_form} from $0$ to $\theta$ gives
\begin{equation}\label{eq:normal_form_integral}
\rho_i(\theta;h)=h+\sum_{j=2}^{N}g_j^{(i)}\int_0^\theta \rho_i(s;h)^j\,ds+O(h^{N+1}),\quad i=1,2.
\end{equation}
Write
\begin{equation}\label{eq:rho_expansion}
    \rho_i(\theta;h)=h+\sum_{j=2}^{k}u_j^{(i)}(\theta)h^j+O(h^{k+1}),\quad u_j^{(i)}(0)=0.
\end{equation}
Substitution of \eqref{eq:rho_expansion} into \eqref{eq:normal_form_integral} and comparison of the coefficients of $h^j$ yield
\begin{equation}\label{eq:uj_triangular}
    u_j^{(i)}(\theta)=\theta g_j^{(i)}+Q_j\left(\theta;g_2^{(i)},\cdots,g_{j-1}^{(i)}\right),\quad Q_2=0,
\end{equation}
where $Q_j$ depends only on the lower-order normal-form coefficients. By Definition \ref{def:Vk_poincare}, the successive function is
\begin{equation}\label{eq:successive_from_rho}
    \Delta(h)=\rho_1(\pi;h)-\rho_2(\pi;h)=\sum_{j=2}^{k}V_jh^j+O(h^{k+1}).
\end{equation}
It follows from \eqref{eq:uj_triangular} and \eqref{eq:successive_from_rho} that
\begin{align}
    V_j={}&\pi\left(g_j^{(1)}-g_j^{(2)}\right)+Q_j\left(\pi;g_2^{(1)},\cdots,g_{j-1}^{(1)}\right) \notag\\
    &-Q_j\left(\pi;g_2^{(2)},\cdots,g_{j-1}^{(2)}\right).
\label{eq:Vj_triangular}
\end{align}
Since $Q_2=0$, the conditions $V_2=\cdots=V_{k-1}=0$ imply successively
\begin{equation}\label{eq:equal_lower_g}
    g_j^{(1)}=g_j^{(2)},\quad j=2,\cdots,k-1.
\end{equation}
Setting $j=k$ in \eqref{eq:Vj_triangular} and using \eqref{eq:equal_lower_g} proves \eqref{eq:Vk_normal_form}.
\end{proof}

The following lemma gives sufficient conditions for the existence of small-amplitude limit cycles in the switching system \eqref{sys:2.5}.

\begin{lemma}\cite{han2012normal}\label{limit cycle}
    Assume $V_j$ depends on $k$ parameters, i.e., 
    \begin{equation*}
        V_j=c_j(\varepsilon_1,\varepsilon_2,\cdots,\varepsilon_k),\quad j=0,1,\cdots,k.
    \end{equation*}
    Assume $c_j(\mathbf{0})=0$, $j=0,1,\cdots,k-1$, $c_k(\mathbf{0})\neq 0$ and 
    \begin{equation*}
        det[\frac{\partial(c_0,c_1,\cdots,c_{k-1})}{\partial(\varepsilon_1,\varepsilon_2,\cdots,\varepsilon_k)}(\mathbf{0})]\neq 0.
    \end{equation*}
    Then for any given $\varepsilon_0>0$, there exist $\varepsilon_1,\varepsilon_2,\cdots,\varepsilon_k$, and $\delta>0$ with $|\varepsilon_j|<\varepsilon_0$, $j=1,2,\cdots,k$, such that the system has exactly $k$ limit cycles.
\end{lemma}

The center problem in switching systems is more complicated than that of smooth systems. The following symmetry criterion will be used to identify center conditions.
\begin{lemma}\label{lem:symmetry_center}\cite{li2015center}
    If system \eqref{sys:2.5} is symmetric with respect to the $x$-axis, i.e. the functions on the right-hand side of system \eqref{sys:2.5} satisfy
    \begin{equation}
        X_k^{(1)}(x, y) = -X_k^{(2)}(x, -y), \quad Y_k^{(1)}(x, y) = Y_k^{(2)}(x, -y),
    \end{equation}
    or if system \eqref{sys:2.5} is symmetric with respect to the $y$-axis, i.e. the functions on the right-hand side of system \eqref{sys:2.5} satisfy
    \begin{equation}
        \begin{aligned}  
            X_k^+(x, y) &= X_k^+(-x, y), & X_k^-(x, y) &= X_k^-(-x, y), \\  
            Y_k^+(x, y) &= -Y_k^+(-x, y), & Y_k^-(x, y) &= -Y_k^-(-x, y),  
        \end{aligned}
    \end{equation}
    then the origin of system \eqref{sys:2.5} is a center.
\end{lemma}

\begin{definition}\label{def:laurent_polynomial}
A \emph{Laurent polynomial} in the variable $z$ over $\mathbb{C}$ is a finite sum of the form
\[
    P(z)=\sum_{k=m}^{n} c_k z^k,
\]
where $m,n\in\mathbb{Z}$ with $m\leq n$ and $c_k\in\mathbb{C}$. The set of all Laurent polynomials is denoted by $\mathbb{C}[z,z^{-1}]$.
\end{definition}

\section{Main Results}\label{Main Results}
This section develops a symbolic algorithm for computing Lyapunov constants of planar switching systems.                         

\subsection{Lyapunov Constants of the Planar Switching System}
To accommodate the terms generated by the recursive normal-form transformations, we introduce the extended Laurent algebra
\begin{equation*}\label{eq:extended_algebra}
    \mathcal{A}:=\mathbb{C}[\theta][z,z^{-1}]=\mathbb{C}[\theta,z,z^{-1}],\quad z=e^{\mathrm{i}\theta}.
\end{equation*}
For $F(\theta,z)\in\mathcal{A}$, differentiation along $z=e^{\mathrm{i}\theta}$ is represented by
\begin{equation}\label{eq:total_derivation}
    \mathcal{D}:=\frac{\partial}{\partial\theta}+\mathrm{i}z\frac{\partial}{\partial z}.
\end{equation}
Then
\begin{equation}\label{eq:total_derivative_identity}
\frac{d}{d\theta}F\left(\theta,e^{\mathrm{i}\theta}\right)
=\left.\mathcal{D}F(\theta,z)\right|_{z=e^{\mathrm{i}\theta}}.
\end{equation}

\begin{definition}\label{def:operator_M}
Let
\begin{equation*}
    F(\theta,z)=\sum_{p=0}^{P}\sum_{j=m}^{n}c_{p,j}\theta^p z^j\in\mathcal{A}.
\end{equation*}
The map $\mathcal{M}:\mathcal{A}\to\mathbb{C}$ defined by
\begin{equation*}\label{eq:extended_mean_operator}
    \mathcal{M}[F]:=\sum_{p=0}^{P}\sum_{j=m}^{n}c_{p,j}\mu_{p,j}
\end{equation*}
is called the algebraic half-period mean operator, where
\begin{equation*}\label{eq:mean_basis_zero}
    \mu_{p,0}:=\frac{\pi^p}{p+1},
\end{equation*}
and for $j\neq0$,
\begin{equation*}\label{eq:mean_basis_nonzero}
    \mu_{0,j}:=\frac{(-1)^j-1}{j\mathrm{i}\pi},\quad \mu_{p,j}:=\frac{(-1)^j\pi^{p-1}}{j\mathrm{i}}-\frac{p}{j\mathrm{i}}\mu_{p-1,j},\qquad p\geq1.
\end{equation*}
\end{definition}

\begin{definition}\label{def:operator_I}
Define the linear operator $\mathcal I:\mathcal A\to\mathcal A$ by
\begin{equation}\label{eq:extended_integration_operator}
    \mathcal{I}[F]:=\sum_{p=0}^{P}\sum_{j=m}^{n}c_{p,j}\mathcal{I}\!\left[\theta^p z^j\right],
\end{equation}
where
\begin{equation}\label{eq:integration_basis}
    \mathcal{I}\!\left[\theta^p z^j\right]:=
\begin{cases}
\dfrac{\theta^{p+1}}{p+1}, & j=0,\\
z^j\displaystyle\sum_{q=0}^{p}\dfrac{(-1)^q p!}{(p-q)!(j\mathrm{i})^{q+1}}\theta^{p-q}-\dfrac{(-1)^p p!}{(j\mathrm{i})^{p+1}}, & j\neq0.
\end{cases}
\end{equation}
\end{definition}

Consider the trigonometric polynomial
\begin{equation}
    R(\theta)=a_0+\sum_{k=1}^{n}\left(a_k\cos k\theta+b_k\sin k\theta\right).
\end{equation}
Under the substitution $z=e^{\mathrm{i}\theta}$, it can be represented as the Laurent polynomial
\begin{equation}
    P(z)=\sum_{k=-n}^{n}c_kz^k\in\mathbb{C}[z,z^{-1}]\subset\mathcal{A},
\end{equation}
where
\begin{equation*}
    c_0=a_0,\quad c_k=\frac{a_k-\mathrm{i}b_k}{2},\quad c_{-k}=\frac{a_k+\mathrm{i}b_k}{2},\quad 1\leq k\leq n.
\end{equation*}

\begin{lemma}\label{lemma:algebraic_equivalence}
Let
\begin{equation}\label{eq:F(theta,z))}
    F(\theta,z)=\sum_{p=0}^{P}\sum_{j=m}^{n}c_{p,j}\theta^p z^j\in\mathcal{A}.
\end{equation}
Then
\begin{equation}\label{eq:extended_algebraic_equivalence}
    \mathcal{M}[F]=\frac{1}{\pi}\int_0^\pi F\left(\theta,e^{\mathrm{i}\theta}\right)\,d\theta.
\end{equation}
\end{lemma}

\begin{proof}
Substituting $z=e^{\mathrm{i}\theta}$ into \eqref{eq:F(theta,z))} gives
\begin{equation}
    F\left(\theta,e^{\mathrm{i}\theta}\right)=\sum_{p=0}^{P}\sum_{j=m}^{n}c_{p,j}\theta^p e^{j\mathrm{i}\theta}.
\end{equation}
Hence 
\begin{equation}\label{eq:mean_expansion}
    \frac{1}{\pi}\int_0^\pi F\left(\theta,e^{\mathrm{i}\theta}\right)\,d\theta
    =\sum_{p=0}^{P}\sum_{j=m}^{n}c_{p,j}I_{p,j}, \quad
    I_{p,j}:=\frac{1}{\pi}\int_0^\pi\theta^p e^{j\mathrm{i}\theta}\,d\theta.
\end{equation}
Direct calculation and integration by parts give
\begin{equation}\label{eq:I_pj}
I_{p,j}=
\begin{cases}
\dfrac{\pi^p}{p+1}, & j=0,\\[2mm]
\dfrac{(-1)^j-1}{j\mathrm{i}\pi}, & j\neq0,\ p=0,\\[2mm]
\dfrac{(-1)^j\pi^{p-1}}{j\mathrm{i}}-\dfrac{p}{j\mathrm{i}}I_{p-1,j}, & j\neq0,\ p\geq1.
\end{cases}
\end{equation}
Therefore,
\begin{equation}
\begin{split}
\frac{1}{\pi}\int_0^\pi F\left(\theta,e^{\mathrm{i}\theta}\right)\,d\theta
&=\sum_{p=0}^{P}\sum_{j=m}^{n}c_{p,j}I_{p,j}\\
&=\sum_{p=0}^{P}\sum_{j=m}^{n}c_{p,j}\mu_{p,j}\\
&=\mathcal{M}[F].
\end{split}
\end{equation}
\end{proof}

\begin{lemma}\label{lemma:integration_operator}
For every $F\in\mathcal{A}$, the algebraic integration operator satisfies
\begin{equation}\label{eq:integration_operator_properties}
    \mathcal{D}\mathcal{I}[F]=F,\quad
    \mathcal{I}[F](0,1)=0,\quad
    \mathcal{I}[F](\pi,-1)=\pi\mathcal{M}[F].
\end{equation}
\end{lemma}

\begin{proof}
The first two identities follow directly from Definition~\ref{def:operator_I}. Along $z=e^{\mathrm{i}\theta}$, the chain rule gives
\begin{equation}
\begin{split}
    \frac{d}{d\theta}\mathcal{I}[F]\left(\theta,z\right)|_{z=e^{\mathrm{i}\theta}}
    &=\bigg(\frac{\partial \mathcal{I}[F]}{\partial\theta}\left(\theta,z\right)
    +\frac{\partial \mathcal{I}[F]}{\partial z}\left(\theta,z\right)
    \frac{d z}{d\theta}\bigg) |_{z=e^{\mathrm{i}\theta}}\\
    &=\bigg(\frac{\partial \mathcal{I}[F]}{\partial\theta}\left(\theta,z\right)
    +\mathrm{i}e^{\mathrm{i}\theta}
    \frac{\partial \mathcal{I}[F]}{\partial z}\left(\theta,z\right)\bigg)|_{z=e^{\mathrm{i}\theta}}\\
    &=\left.
    \left(\frac{\partial}{\partial\theta}
    +\mathrm{i}z\frac{\partial}{\partial z}\right)\mathcal{I}[F](\theta,z)
    \right|_{z=e^{\mathrm{i}\theta}}\\
    &=\left.\mathcal D\mathcal{I}[F](\theta,z)\right|_{z=e^{\mathrm{i}\theta}}\\
    &=F(\theta,e^{\mathrm{i}\theta}).
\end{split}
\end{equation}
Hence,
\begin{align}
    \mathcal{I}[F](\pi,-1)-\mathcal{I}[F](0,1)
    &=\int_0^\pi F\left(\theta,e^{\mathrm{i}\theta}\right)\,d\theta \notag\\
    &=\pi\mathcal{M}[F],
\end{align}
where the last equality follows from Lemma~\ref{lemma:algebraic_equivalence}. Since $\mathcal{I}[F](0,1)=0$, the third identity follows.
\end{proof}

For any fixed $N\geq2$, Lemma~\ref{lemma:normal_form} guarantees the existence of a near-identity transformation
\begin{equation}\label{eq:transform_def}
    r=\rho+\sum_{k=2}^{N}h_k^{(i)}(\theta)\rho^k,
\end{equation}
where
$h_k^{(i)}(0)=h_k^{(i)}(\pi)=0$, 
which transforms the radial equation into
\begin{equation}\label{eq:truncated_normal_form}
    \frac{d\rho}{d\theta}=\sum_{k=2}^{N}g_k^{(i)}\rho^k+O(\rho^{N+1}).
\end{equation}
Differentiating \eqref{eq:near_identity} with respect to $\theta$ and using \eqref{eq:normal_form}, we obtain
\begin{equation}\label{eq:derivative_transformation}
    \frac{dr}{d\theta}
    =\sum_{k=2}^{N}\frac{d h_k^{(i)}(\theta)}{d\theta}\rho^k
    +\left(1+\sum_{j=2}^{N}j h_j^{(i)}(\theta)\rho^{j-1}\right)
    \sum_{m=2}^{N}g_m^{(i)}\rho^m
    +O(\rho^{N+1}).
\end{equation}
Substituting \eqref{eq:near_identity} into \eqref{eq:dr/dtheta} gives
\begin{equation}\label{eq:substituted_radial_equation}
    \frac{dr}{d\theta}
    =\sum_{j=2}^{N}R_j^{(i)}(\theta)
    \left(\rho+\sum_{m=2}^{N}h_m^{(i)}(\theta)\rho^m\right)^j
    +O(\rho^{N+1}).
\end{equation}
For $2\leq k\leq N$, define
\begin{equation}\label{eq:lower_order_terms}
    \mathcal{H}_{<k}^{(i)}(\theta,\rho):=\sum_{\ell=2}^{k-1}h_\ell^{(i)}(\theta)\rho^\ell,\quad
    G_{<k}^{(i)}(\rho):=\sum_{m=2}^{k-1}g_m^{(i)}\rho^m,
\end{equation}
Comparing the coefficients of $\rho^k$ yields
\begin{equation}\label{eq:homological_theta}
    \frac{d h_k^{(i)}(\theta)}{d\theta}=\Gamma_k^{(i)}(\theta)-g_k^{(i)},
\end{equation}
where
\begin{equation}\label{eq:source_term}
    \Gamma_k^{(i)}(\theta):=[\rho^k]\left\{
    \sum_{j=2}^{k}R_j^{(i)}(\theta)
    \left(\rho+\mathcal{H}_{<k}^{(i)}(\theta,\rho)\right)^j
    -\frac{\partial\mathcal{H}_{<k}^{(i)}(\theta,\rho)}{\partial\rho}
    G_{<k}^{(i)}(\rho)
    \right\},
\end{equation}
and $[\rho^k]$ denotes the coefficient of $\rho^k$.

For the algebraic computation, let
$\widehat{h}_k^{(i)}(\theta,z),\widetilde{P}_k^{(i)}(\theta,z)\in\mathcal{A}$
be the algebraic representations of $h_k^{(i)}(\theta)$ and $\Gamma_k^{(i)}(\theta)$, respectively. Thus,
\begin{equation}\label{eq:algebraic_representations}
h_k^{(i)}(\theta)=\widehat{h}_k^{(i)}\left(\theta,e^{\mathrm{i}\theta}\right),\qquad
\Gamma_k^{(i)}(\theta)=\widetilde{P}_k^{(i)}\left(\theta,e^{\mathrm{i}\theta}\right).
\end{equation}
Therefore, the algebraic representation of \eqref{eq:homological_theta} is
\begin{equation}\label{eq:algebraic_homological_equation}
    \mathcal{D}\widehat{h}_k^{(i)}(\theta,z)
    =\widetilde{P}_k^{(i)}(\theta,z)-g_k^{(i)}.
\end{equation}
The following theorem determines $g_k^{(i)}$ and $\widehat{h}_k^{(i)}$ by means of the algebraic operators $\mathcal{M}$ and $\mathcal{I}$.

\begin{theorem}\label{thm:algebraic_normal_form}
Let $\widetilde{P}_k^{(i)}(\theta,z)\in\mathcal{A}$ be the algebraic representation of $\Gamma_k^{(i)}(\theta)$. Then the normal-form coefficient $g_k^{(i)}$ is given by
\begin{equation}\label{eq:thm_gk}
    g_k^{(i)}=\mathcal{M}\left[\widetilde{P}_k^{(i)}\right],
\end{equation}
and the algebraic representation of the transformation function $h_k^{(i)}(\theta)$ is
\begin{equation}\label{eq:thm_hk}
    \widehat{h}_k^{(i)}(\theta,z)
    =\mathcal{I}\left[\widetilde{P}_k^{(i)}(\theta,z)-g_k^{(i)}\right].
\end{equation}
Consequently,
\begin{equation*}\label{eq:actual_transformation_function}
    h_k^{(i)}(\theta)=\widehat{h}_k^{(i)}\left(\theta,e^{\mathrm{i}\theta}\right),
\end{equation*}
and
\begin{equation*}\label{eq:transformation_boundary_conditions}
    h_k^{(i)}(0)=h_k^{(i)}(\pi)=0.
\end{equation*}
\end{theorem}

\begin{proof}
Integrating \eqref{eq:homological_theta} over $[0,\pi]$ gives
\begin{equation}\label{eq:integrated_homological_equation}
    h_k^{(i)}(\pi)-h_k^{(i)}(0)=\int_0^\pi\Gamma_k^{(i)}(\theta)\,d\theta-\pi g_k^{(i)}.
\end{equation}
By $h_k^{(i)}(0)=h_k^{(i)}(\pi)=0$, we obtain
\begin{equation}\label{eq:gk_integral_average}
    g_k^{(i)}=\frac{1}{\pi}\int_0^\pi\Gamma_k^{(i)}(\theta)\,d\theta.
\end{equation}
By Lemma~\ref{lemma:algebraic_equivalence},
\begin{equation}
    g_k^{(i)}=\frac{1}{\pi}\int_0^\pi \widetilde{P}_k^{(i)}\left(\theta,e^{\mathrm{i}\theta}\right)\,d\theta
    =\mathcal{M}\left[\widetilde{P}_k^{(i)}\right],
\end{equation}
which proves \eqref{eq:thm_gk}.

By \eqref{eq:algebraic_homological_equation}, we obtain
\begin{equation}\label{eq:general_homological_solution}
    \widehat{h}_k^{(i)}(\theta,z)
    =\mathcal{I}\left[\widetilde{P}_k^{(i)}(\theta,z)-g_k^{(i)}\right]+C_k^{(i)}.
\end{equation}
Evaluating \eqref{eq:general_homological_solution} at $(\theta,z)=(0,1)$ yields
\begin{equation}
    0=h_k^{(i)}(0)=\widehat{h}_k^{(i)}(0,1)
    =\mathcal{I}\left[\widetilde{P}_k^{(i)}(\theta,z)-g_k^{(i)}\right](0,1)+C_k^{(i)}
    =C_k^{(i)}.
\end{equation}
Therefore,
\begin{equation}
    \widehat{h}_k^{(i)}(\theta,z)
    =\mathcal{I}\left[
    \widetilde{P}_k^{(i)}(\theta,z)-g_k^{(i)}
    \right],
\end{equation}
which proves \eqref{eq:thm_hk}.
\end{proof} 

\subsection{Recursive Algorithms}

We now present a recursive algorithm for calculating Lyapunov constants of planar switching systems.

\begin{breakablealgorithm}
    \caption{Lyapunov Constants for the Planar Switching System \eqref{sys:2.5}}
    \label{alg:algebraic_nf}
    \begin{algorithmic}[1]
        \REQUIRE The two subsystems $S^{(1)}$ and $S^{(2)}$ in the form
        \[
        \dot{x}=-y+P^{(i)}(x,y),\qquad
        \dot{y}=x+Q^{(i)}(x,y),\qquad i=1,2,
        \]
        and the maximum order $n$.
        \ENSURE Lyapunov constants
        $V=\{V_2,\ldots,V_n\}$.

        \STATE $x\leftarrow\frac{r}{2}(z+z^{-1})$,
        $y\leftarrow\frac{r}{2\mathrm{i}}(z-z^{-1})$,
        where $z=e^{\mathrm{i}\theta}$

        \FOR{$i\in\{1,2\}$}
            \STATE
            $\displaystyle
            R^{(i)}(r,z)
            \leftarrow
            r\frac{xP^{(i)}(x,y)+yQ^{(i)}(x,y)}
            {r^2+xQ^{(i)}(x,y)-yP^{(i)}(x,y)}$

            \STATE
            $\displaystyle
            R^{(i)}(r,z)
            \leftarrow
            \operatorname{Trunc}_{r}^{n}
            R^{(i)}(r,z)$

            \STATE
            $\widehat{\mathcal{H}}_{<2}^{(i)}(\theta,z,\rho)\leftarrow0$

            \STATE
            $G_{<2}^{(i)}(\rho)\leftarrow0$
        \ENDFOR

        \STATE $\mathcal{L}\leftarrow[\,]$

        \FOR{$k=2$ \TO $n$}
            \FOR{$i\in\{1,2\}$}
                \STATE
                $\displaystyle
                F_k^{(i)}(\theta,z,\rho)
                \leftarrow
                \operatorname{Trunc}_{\rho}^{k}
                \left[
                R^{(i)}
                \left(
                \rho+\widehat{\mathcal{H}}_{<k}^{(i)}(\theta,z,\rho),z
                \right)
                -
                \frac{\partial
                \widehat{\mathcal{H}}_{<k}^{(i)}(\theta,z,\rho)}
                {\partial\rho}
                G_{<k}^{(i)}(\rho)
                \right]$

                \STATE
                $\displaystyle
                \widetilde{P}_k^{(i)}(\theta,z)
                \leftarrow
                [\rho^k]F_k^{(i)}(\theta,z,\rho)$

                \STATE
                $\displaystyle
                g_k^{(i)}
                \leftarrow
                \mathcal{M}
                \left[\widetilde{P}_k^{(i)}(\theta,z)\right]$

                \IF{$k<n$}
                    \STATE
                    $\displaystyle
                    Q_k^{(i)}(\theta,z)
                    \leftarrow
                    \widetilde{P}_k^{(i)}(\theta,z)-g_k^{(i)}$

                    \STATE
                    $\displaystyle
                    \widehat{h}_k^{(i)}(\theta,z)
                    \leftarrow
                    \mathcal{I}
                    \left[Q_k^{(i)}(\theta,z)\right]$

                    \STATE
                    $\displaystyle
                    \widehat{\mathcal{H}}_{<k+1}^{(i)}(\theta,z,\rho)
                    \leftarrow
                    \widehat{\mathcal{H}}_{<k}^{(i)}(\theta,z,\rho)
                    +
                    \widehat{h}_k^{(i)}(\theta,z)\rho^k$

                    \STATE
                    $\displaystyle
                    G_{<k+1}^{(i)}(\rho)
                    \leftarrow
                    G_{<k}^{(i)}(\rho)+g_k^{(i)}\rho^k$
                \ENDIF
            \ENDFOR

            \STATE
            $\displaystyle
            V_k
            \leftarrow
            \pi\left(g_k^{(1)}-g_k^{(2)}\right)$

            \STATE
            \STATE $V \leftarrow V \cup \{V_k\}$
        \ENDFOR

        \RETURN $V$
    \end{algorithmic}
\end{breakablealgorithm}

\section{Examples and Experiments}\label{Examples}
In this section, we present two examples to illustrate the application and performance of the proposed algorithm. 

\subsection{Example 1}\label{example:1}
Consider a generalized switching quartic Liénard system
\begin{equation} \label{ex:1}
    \begin{pmatrix}
        \dot{x} \\
        \dot{y}
    \end{pmatrix}=
    \begin{cases}
    \begin{pmatrix}
        \delta x+y - (a_{21}x^2 + a_{31}x^3+a_{41}x^4) \\
        -x - (b_{21}x^2 + b_{31}x^3 + b_{41}x^4)
    \end{pmatrix}, & \text{if } y \ge 0, \\
    \begin{pmatrix}
        \delta x + y - (a_{22}x^2 + a_{32}x^3 + a_{42}x^4) \\
        -x- (b_{22}x^2 + b_{32}x^3 + b_{42}x^4)
    \end{pmatrix}, & \text{if } y < 0.
    \end{cases}
\end{equation}

\begin{theorem}\label{thm:center_conditions}
The system \eqref{ex:1} has a center at the orign if and only if the following conditions holds:
\begin{align}
    C_{\mathrm{I}}:\quad &\delta=0,\quad a_{21}+a_{22}=a_{31}+a_{32}=a_{41}+a_{42}=0,
    \notag\\
    &b_{21}-b_{22}=b_{31}-b_{32}=b_{41}-b_{42}=0;
    \label{eq:center_condition_I}\\
    C_{\mathrm{II}}:\quad
    &\delta=b_{21}=b_{22}=a_{31}=a_{32}=b_{41}=b_{42}=0.
    \label{eq:center_condition_II}
\end{align}
\end{theorem}

\begin{theorem}\label{thm:ten_limit_cycles}
    For suitable parameter values and arbitrarily small perturbations, the system \eqref{ex:1} has exactly ten small-amplitude limit cycles bifurcating from the origin. 
\end{theorem}

\begin{proof}
We apply the coordinate transformation $y \to -y$ and time reversal $t \to -t$ to the lower system ($y<0$). Then \eqref{ex:1} is transformed into
\begin{align}
    \left\{
    \begin{aligned}
        \dot{x} &= -\delta x + y + (a_{22}x^2 + a_{32}x^3 + a_{42}x^4), \\
        \dot{y} &= -x- (b_{22}x^2 + b_{32}x^3 + b_{42}x^4).
    \end{aligned}
    \right.
\end{align}
Subsequently, the first ten Lyapunov constants can be computed using Algorithm \ref{alg:algebraic_nf}. Setting $\delta=0$ yields $V_1=0$. Then $V_2=\frac{2}{3}(b_{21}-b_{22})=0$ leads to $b_{21}=b_{22}$. Then
\begin{equation*}
    V_3=-\frac{(2a_{21}b_{22} + 2a_{22}b_{22} - 3a_{31} - 3a_{32})\pi}{8},
\end{equation*}
which leads to 
\begin{equation*}
    a_{32} = \frac{2a_{21}b_{22}}{3} + \frac{2a_{22}b_{22}}{3} - a_{31}.
\end{equation*}
Then
\begin{equation*}
    V_4=-\frac{28}{45}a_{21}^2 b_{22} - \frac{2}{3}b_{22}b_{31} + \frac{2}{3}b_{22}b_{32} + \frac{14}{15}a_{21}a_{31} - \frac{28}{45}a_{21}a_{22}b_{22} + \frac{14}{15}a_{22}a_{31} + \frac{2}{5}b_{41} - \frac{2}{5}b_{42},
\end{equation*}
and setting $V_4=0$ provides
\begin{equation*}
    b_{42} = -\frac{14}{9}a_{21}^2 b_{22} - \frac{5}{3}b_{22}b_{31} + \frac{5}{3}b_{22}b_{32} + \frac{7}{3}a_{21}a_{31} - \frac{14}{9}a_{21}a_{22}b_{22} + \frac{7}{3}a_{22}a_{31} + b_{41}.
\end{equation*}
Similarly, we obtain
\begin{equation*}
\begin{aligned}
    V_5 &= -\frac{\pi}{576} \bigg( 104a_{21}^3b_{22} - 112a_{21}^2a_{22}b_{22} - 216a_{21}a_{22}^2b_{22} - 156a_{21}^2a_{31} + 168a_{21}a_{22}a_{31} \\
    &\quad - 210a_{21}b_{22}b_{31} + 90a_{21}b_{22}b_{32} + 324a_{22}^2a_{31} - 120a_{22}b_{22}b_{31} + 72a_{21}b_{41} + 72a_{22}b_{41} \\
    &\quad + 135a_{31}b_{31} - 135a_{31}b_{32} + 240a_{41}b_{22} + 240a_{42}b_{22} \bigg)
\end{aligned} 
\end{equation*}

\textbf{Case i}: $b_{22}=0$. We have
\begin{align*}
    V_5 &= \frac{\pi}{192} \big( 52a_{21}^2a_{31} - 56a_{21}a_{22}a_{31} - 108a_{22}^2a_{31} - 24a_{21}b_{41} \\
    & - 24a_{22}b_{41} - 45a_{31}b_{31} + 45a_{31}b_{32} \big).
\end{align*}

\textbf{Case i.1}: $a_{21}+a_{22}=0$. We have
\begin{equation*}
    V_5 = -\frac{15 a_{31} (b_{31} - b_{32}) \pi}{64}.
\end{equation*}
Setting $V_5=0$ requires $b_{31} = b_{32}$, which yields
\begin{equation*}
    V_6 = \frac{58}{525}a_{31}(a_{41} + a_{42}).
\end{equation*}
Setting $V_6=0$, we obtain $a_{41} = -a_{42}$.
In this case, the complete set of conditions derived thus far can be simplified to
\begin{equation}\label{condition1}
    a_{21} = -a_{22}, \quad a_{31} = -a_{32}, \quad a_{41} = -a_{42}, \quad b_{21} = b_{22}, \quad b_{31} = b_{32}, \quad b_{41} = b_{42}.
\end{equation}

By \eqref{condition1}, the switching system \eqref{ex:1} can be rewritten as
\begin{equation}\label{case i.1}
    \begin{pmatrix}
        \dot{x} \\
        \dot{y}
    \end{pmatrix}=
    \begin{cases}
    \begin{pmatrix}
        y - (a_{21}x^2 + a_{31}x^3+a_{41}x^4) \\
        -x - (b_{21}x^2 + b_{31}x^3 + b_{41}x^4)
    \end{pmatrix}, & \text{if } y \ge 0, \\[3ex]
    \begin{pmatrix}
        y + (a_{21}x^2 + a_{31}x^3 + a_{41}x^4) \\
        -x - (b_{21}x^2 + b_{31}x^3 + b_{41}x^4)
    \end{pmatrix}, & \text{if } y < 0.
    \end{cases}
\end{equation}
By Lemma \ref{lem:symmetry_center}, it follows that the origin of system \eqref{ex:1} is a center.

\textbf{Case i.2}: $a_{21}+a_{22} \neq 0$. From $V_5=0$, we have
\begin{equation}\label{eq:b_{41}}
    b_{41}=\frac{a_{31} (52a_{21}^2 - 56a_{21}a_{22} - 108a_{22}^2 - 45b_{31} + 45b_{32})}{24(a_{21} + a_{22})}.
\end{equation}
By \eqref{eq:b_{41}}, we obtain
\begin{equation*}
\begin{aligned}
    V_6 &= -\frac{a_{31}}{18900(a_{21} + a_{22})} \bigg( 4336a_{21}^4 - 37784a_{21}^3a_{22} - 84240a_{21}^2a_{22}^2 - 37784a_{21}a_{22}^3 + 4336a_{22}^4 \\
    &\quad + 13014a_{21}^2b_{31} - 16470a_{21}^2b_{32} - 3456a_{21}a_{22}b_{31} - 3456a_{21}a_{22}b_{32} - 16470a_{22}^2b_{31} \\
    &\quad + 13014a_{22}^2b_{32} - 2088a_{21}a_{41} - 2088a_{21}a_{42} - 2088a_{22}a_{41} - 2088a_{22}a_{42} \\
    &\quad - 14175b_{31}^2 + 28350b_{31}b_{32} - 14175b_{32}^2 \bigg).
\end{aligned}
\end{equation*}

\textbf{Case i.2.1}: $a_{31}=0$. In this case, the conditions derived thus far can be simplified to
\begin{equation}
    b_{21}=b_{22}=a_{31}=a_{32}=b_{41}=b_{42}=0.
\end{equation}
Now \eqref{ex:1} can be rewritten as
\begin{equation}
    \begin{pmatrix}
        \dot{x} \\
        \dot{y}
    \end{pmatrix}=
    \begin{cases}
    \begin{pmatrix}
        y - (a_{21}x^2 + a_{41}x^4) \\
        -x - b_{31}x^3
    \end{pmatrix}, & \text{if } y \ge 0, \\
    \begin{pmatrix}
        y - (a_{22}x^2 + a_{42}x^4) \\
        -x - b_{32}x^3
    \end{pmatrix}, & \text{if } y < 0.
    \end{cases}
\end{equation}

To prove that the origin is a center, we examine the symmetry of the vector fields. Applying $(x, y, t) \mapsto (-x, y, -t)$ to both the upper ($i=1$) and lower ($i=2$) subsystems yields
\begin{equation}
    \begin{cases}
        \frac{dx}{dt} = y - (a_{2i}x^2 + a_{41}x^4), \\
        \frac{dy}{dt} = -x - b_{3i}x^3.
    \end{cases}
\end{equation}

By Lemma \ref{lem:symmetry_center}, it follows that the origin of system \eqref{ex:1} is a center.

\textbf{Case ii}: $b_{22} \neq 0$. From $V_5=0$, we have 
\begin{equation*}
\begin{aligned}
    a_{42} &= -\frac{1}{240b_{22}} \bigg( 104a_{21}^3b_{22} - 112a_{21}^2a_{22}b_{22} - 216a_{21}a_{22}^2b_{22} - 156a_{21}^2a_{31} + 168a_{21}a_{22}a_{31} \\
    &\quad - 210a_{21}b_{22}b_{31} + 90a_{21}b_{22}b_{32} + 324a_{22}^2a_{31} - 120a_{22}b_{22}b_{31} \\
    &\quad + 72a_{21}b_{41} + 72a_{22}b_{41} + 135a_{31}b_{31} - 135a_{31}b_{32} + 240a_{41}b_{22} \bigg).
\end{aligned}
\end{equation*}

\textbf{Case ii.1}: $a_{21}+a_{22}=0$. Then
\begin{align*}
    V_6 &= -\frac{b_{31} - b_{32}}{12600b_{22}} \big( 1844a_{22}^2b_{22}^2 - 5600b_{22}^4 + 3288a_{22}a_{31}b_{22} \\
    &- 8400b_{22}^2b_{31} + 783a_{31}^2 + 5040b_{22}b_{41} \big).
\end{align*}
Requiring $V_6 = 0$ directly yields $b_{31} = b_{32}$. Under these parameter conditions derived thus far, the system takes the form
\begin{equation}
    \begin{pmatrix}
        \dot{x} \\
        \dot{y}
    \end{pmatrix}=
    \begin{cases}
    \begin{pmatrix}
        y - (a_{21}x^2 + a_{31}x^3+a_{41}x^4) \\
        -x - (b_{21}x^2 + b_{31}x^3 + b_{41}x^4)
    \end{pmatrix}, & \text{if } y \ge 0, \\[3ex]
    \begin{pmatrix}
        y + (a_{21}x^2 + a_{31}x^3 + a_{41}x^4) \\
        -x - (b_{21}x^2 + b_{31}x^3 + b_{41}x^4)
    \end{pmatrix}, & \text{if } y < 0,
    \end{cases}
\end{equation}
which coincides with \eqref{case i.1} considered in \textbf{Case i.1}. Consequently, the origin is guaranteed to be a center.

\textbf{Case ii.2}: $a_{21}+a_{22} \neq 0$. It follows that
\begin{equation*}
\begin{aligned}
    V_6 &:= -\frac{1}{567000b_{22}} \Big( 300784a_{21}^4b_{22}^2 + 425088a_{21}^3a_{22}b_{22}^2 + 37584a_{21}^2a_{22}^2b_{22}^2 - 336000a_{21}^2b_{22}^4 \\
    &\quad - 86720a_{21}a_{22}^3b_{22}^2 - 336000a_{21}a_{22}b_{22}^4 - 424032a_{21}^3a_{31}b_{22} - 666864a_{21}^2a_{22}a_{31}b_{22} \\
    &\quad - 715140a_{21}^2b_{22}^2b_{31} + 337140a_{21}^2b_{22}^2b_{32} - 112752a_{21}a_{22}^2a_{31}b_{22} - 798120a_{21}a_{22}b_{22}^2b_{31} \\
    &\quad + 420120a_{21}a_{22}b_{22}^2b_{32} + 504000a_{21}a_{31}b_{22}^3 + 130080a_{22}^3a_{31}b_{22} + 504000a_{22}a_{31}b_{22}^3 \\
    &\quad - 252000b_{22}^4b_{31} + 252000b_{22}^4b_{32} - 40716a_{21}^2a_{31}^2 + 268272a_{21}^2b_{22}b_{41} + 43848a_{21}a_{22}a_{31}^2 \\
    &\quad + 268272a_{21}a_{22}b_{22}b_{41} + 347220a_{21}a_{31}b_{22}b_{31} - 482220a_{21}a_{31}b_{22}b_{32} + 936000a_{21}a_{41}b_{22}^2 \\
    &\quad + 84564a_{22}^2a_{31}^2 + 495180a_{22}a_{31}b_{22}b_{31} - 630180a_{22}a_{31}b_{22}b_{32} + 936000a_{22}a_{41}b_{22}^2 \\
    &\quad - 378000b_{22}^2b_{31}^2 + 378000b_{22}^2b_{31}b_{32} + 18792a_{21}a_{31}b_{41} + 18792a_{22}a_{31}b_{41} \\
    &\quad + 35235a_{31}^2b_{31} - 35235a_{31}^2b_{32} + 226800b_{22}b_{31}b_{41} - 226800b_{22}b_{32}b_{41} \Big).
\end{aligned}
\end{equation*}

To ensure  $V_6 = 0$, we obtain
\begin{equation*}
\begin{aligned}
    a_{41} &= -\frac{1}{936000b_{22}^2(a_{21} + a_{22})} \Big( 300784a_{21}^4b_{22}^2 + 425088a_{21}^3a_{22}b_{22}^2 + 37584a_{21}^2a_{22}^2b_{22}^2 \\
    &\quad - 336000a_{21}^2b_{22}^4 - 86720a_{21}a_{22}^3b_{22}^2 - 336000a_{21}a_{22}b_{22}^4 - 424032a_{21}^3a_{31}b_{22} \\
    &\quad - 666864a_{21}^2a_{22}a_{31}b_{22} - 715140a_{21}^2b_{22}^2b_{31} + 337140a_{21}^2b_{22}^2b_{32} - 112752a_{21}a_{22}^2a_{31}b_{22} \\
    &\quad - 798120a_{21}a_{22}b_{22}^2b_{31} + 420120a_{21}a_{22}b_{22}^2b_{32} + 504000a_{21}a_{31}b_{22}^3 + 130080a_{22}^3a_{31}b_{22} \\
    &\quad + 504000a_{22}a_{31}b_{22}^3 - 252000b_{22}^4b_{31} + 252000b_{22}^4b_{32} - 40716a_{21}^2a_{31}^2 \\
    &\quad + 268272a_{21}^2b_{22}b_{41} + 43848a_{21}a_{22}a_{31}^2 + 268272a_{21}a_{22}b_{22}b_{41} + 347220a_{21}a_{31}b_{22}b_{31} \\
    &\quad - 482220a_{21}a_{31}b_{22}b_{32} + 84564a_{22}^2a_{31}^2 + 495180a_{22}a_{31}b_{22}b_{31} - 630180a_{22}a_{31}b_{22}b_{32} \\
    &\quad - 378000b_{22}^2b_{31}^2 + 378000b_{22}^2b_{31}b_{32} + 18792a_{21}a_{31}b_{41} + 18792a_{22}a_{31}b_{41} \\
    &\quad + 35235a_{31}^2b_{31} - 35235a_{31}^2b_{32} + 226800b_{22}b_{31}b_{41} - 226800b_{22}b_{32}b_{41} \Big).
\end{aligned}
\end{equation*}

Insertion of the derived conditions simplifies $V_7$ and $V_8$ into the fractional forms
\begin{equation}
    V_7 = \frac{M_7}{A_1}, \quad V_8 = \frac{M_8}{A_2}.
\end{equation}
The requirement for these constants to vanish dictates that their numerators must be zero, which establishes
\begin{equation}
    M_7 = 0, \quad M_8 = 0.
\end{equation}
As detailed in \ref{appendix_1}, $M_7$ and $M_8$ are quadratic polynomials in $b_{41}$, in the form of
\begin{equation}
\begin{aligned}
    M_7 &= p b_{41}^2 + c_1 b_{41} + c_0 = 0, \\
    M_8 &= q b_{41}^2 + d_1 b_{41} + d_0 = 0,
\end{aligned}
\end{equation}
where the coefficients $p, q$ and $c_i, d_i$ ($i=0,1$) are independent of $b_{41}$. 

Eliminating $b_{41}^2$ from $M_7$ and $M_8$ yields
\begin{equation}
    q M_7 - p M_8 = 0.
\end{equation}
The combination $qM_7-pM_8$ is linear in $b_{41}$. Solving this equation gives the expression for $b_{41}$ listed in \ref{appendix_1}.

Substituting the derived expression for $b_{41}$ into $M_8$ and enforcing the condition $M_8 = 0$ allows us to express $b_{32}$ as
\begin{equation*}
\begin{aligned}
    b_{32} &= \frac{1}{34977600b_{22}} \big(32248125\pi a_{21}b_{22}^2 + 32248125\pi a_{22}b_{22}^2 + 53395904a_{21}^2b_{22} \\
    &\quad - 1932096a_{21}a_{22}b_{22} - 55328000a_{22}^2b_{22} + 2898144a_{21}a_{31} \\
    &\quad + 2898144a_{22}a_{31} + 34977600b_{22}b_{31} \big).
\end{aligned}
\end{equation*}

Further evaluation of the higher-order Lyapunov constants under the established constraints gives
\begin{equation}
    V_9 = \frac{b_{22}(a_{21}+a_{22})N_9}{A_3}, \quad V_{10} = \frac{b_{22}(a_{21}+a_{22})N_{10}}{A_4}.
\end{equation}
$V_9 = V_{10} = 0$ hold if and only if their simplified numerators vanish, which reduces to
\begin{equation}
    N_9 = 0, \quad N_{10} = 0.
\end{equation}

To eliminate $a_{21}$, we compute the resultant of $N_9$ and $N_{10}$ with respect to this variable. This produces a polynomial condition solely in terms of $a_{22}$ as
\begin{equation}
\text{Res}[N_9, N_{10}, a_{21}] = 0.
\end{equation}

To alleviate the computational complexity, we fix the parameters $b_{22} = 1$, $a_{31} = 0$, and $b_{31} = 0$. Numerical evaluation of this resultant equation yields a finite set of roots for $a_{22}$. Among the real roots, we select the root $a_{22}^*$ and determine its high-precision numerical value to be
\begin{equation}
\begin{aligned}
    a_{22}^* &= -1.10537021875498463013092970083802084162514868316355870327048 \\
    &\quad \ 46251431085942888875905928714411556190626455050653412342160 \\
    &\quad \ 6865036271014821217356391432835\cdots.
\end{aligned}
\end{equation}

Substituting $a_{22}^*$ back into the reduced equation $N_9 = 0$ determines the corresponding numerical value for $a_{21}$ as
\begin{equation}
\begin{aligned}
    a_{21}^* &= 4.58271329818940593699661949857618500569500045610176400170278 \\
    &\quad \ 13020411921968051837373202834427635590047173038374801943219 \\
    &\quad \ 4680293633645973224575027897834\cdots.
\end{aligned}
\end{equation}
Revisiting the original parametric constraints, the numerical values for the remaining parameters are evaluated as
\begin{equation}
\begin{split}
    b_{21}^*&=1, \quad a_{31}^*=0, \quad b_{31}^*=0\\
        a_{32}^*&=2.318228719622947537910459865158776109379901181958803532288197 \\
    &\quad \ 784598722401677530764484941334405293294714532514759306737252 \\
    &\quad \ 10171575087434671479090976666\cdots, \\
        b_{32}^* &= 40.478982676093856309103682772064563391437458289104495 \\
    &\quad \ 83034216611904657854591836882213733914253566901088281671086360 \\
    &\quad \ 87766306328166318324424062262785372\cdots, \\
    b_{41}^* &= 39.646257687264953592871622019448492361257253578860506 \\
    &\quad \ 42966012166024241207457852990508622124588381652537951880786901 \\
    &\quad \ 03842378464147971868526364866221924\cdots, \\
    a_{41}^* &= -30.65038246861737697127373707318315115855568542205338 \\
    &\quad \ 02056816728576422001066461877180666562570460729035000472800784 \\
    &\quad \ 319326401025663603283456969455774999\cdots, \\
    a_{42}^* &= -127.7716013893247382242186456425934011764805313596639 \\
    &\quad \ 67400668179006193457701452804015388747211751146399638097486412 \\
    &\quad \ 842353258651791575831382292803926111\cdots, \\
    b_{42}^* &= 82.322414456879790582958103263244302174522319705978300 \\
    &\quad \ 53974202755979003867464221767910976003122817476468431202243766 \\
    &\quad \ 86997050326286369483244824498831635\cdots.
\end{split}
\end{equation}
    
At this parameter point, numerical evaluation gives the ninth, tenth, and eleventh Lyapunov constants as
\begin{align*}
    V_9 &= -1.55318714584841318374461767468566181151703488337346 \\
    &\quad \ 13751921091469497559142772726772139526882514824804 \\
    &\quad \ 5440667356720764690393793468899794720216431910870\cdots \times 10^{-146}, \\
    V_{10} &= 2.56373546354466057761014708713258910211795314556439 \\
    &\quad \ 71482971225366605105959117864768190884513632748350 \\
    &\quad \ 1269682515822525263720001620200014016920176319048\cdots \times 10^{-132}, \\
    V_{11} &= 4.05420776901882766646139421386345550569240668061227 \\
    &\quad \ 64623898978466709060591392999087182599590196314218 \\
    &\quad \ 7007532338273473386377519893113159637494700116872\cdots \times 10^{6}.
\end{align*}
Numerical evaluation with the specified parameters confirms that the lower-order Lyapunov constants vanish within computational precision limits, yielding $V_i \approx 0$ for $i=1, \dots, 10$ and $V_{11} \neq 0$. To apply Lemma~\ref{limit cycle}, we take
\[
\boldsymbol{\lambda}=(\delta,b_{21},a_{32},b_{42},a_{42},a_{41},
b_{41},b_{32},a_{21},a_{22}).
\]
The parameters $\delta$, $b_{21}$, $a_{32}$, $b_{42}$, $a_{42}$, and $a_{41}$ can be used successively to perturb the first six Lyapunov constants. We therefore consider the reduced Jacobian matrix
\begin{equation}\label{eq:reduced_jacobian}
    J_r=\left.\frac{\partial(V_7,V_8,V_9,V_{10})}{\partial(b_{41},b_{32},a_{21},a_{22})}
    \right|_{\boldsymbol{\lambda}=\boldsymbol{\lambda}^*} =-1.2288661497 \times 10^{15} \neq 0.
\end{equation}

Therefore, the hypotheses of Lemma~\ref{limit cycle} are satisfied. Hence, under suitable small perturbations of the ten parameters, system~\eqref{example:1} has 10 small-amplitude limit cycles bifurcating from the origin. This finding establishes a new lower bound for the maximum number of limit cycles in the generalized switching quartic Liénard systems.
\end{proof}

\subsection{Application}
The Alpazur oscillator introduced in \cite{kousaka1999bifurcation} provides a representative switched nonlinear circuit model. In the original system, the two switching states are distinguished by the linear conductance and source terms, whereas the cubic voltage--current characteristic of the nonlinear conductor remains unchanged. In practical switching circuits, different conducting branches may exhibit different nonlinear voltage--current characteristics. Motivated by this observation, we consider a modified Alpazur oscillator in which the $LC$ oscillatory core is retained, while the switch selects between two nonlinear conductor branches. Thus, the modification extends the original switching mechanism to
branch-dependent nonlinear conductor characteristics and allows us to investigate their influence on periodic oscillations.

\begin{figure}[t]
    \centering
    \includegraphics[width=0.75\linewidth]{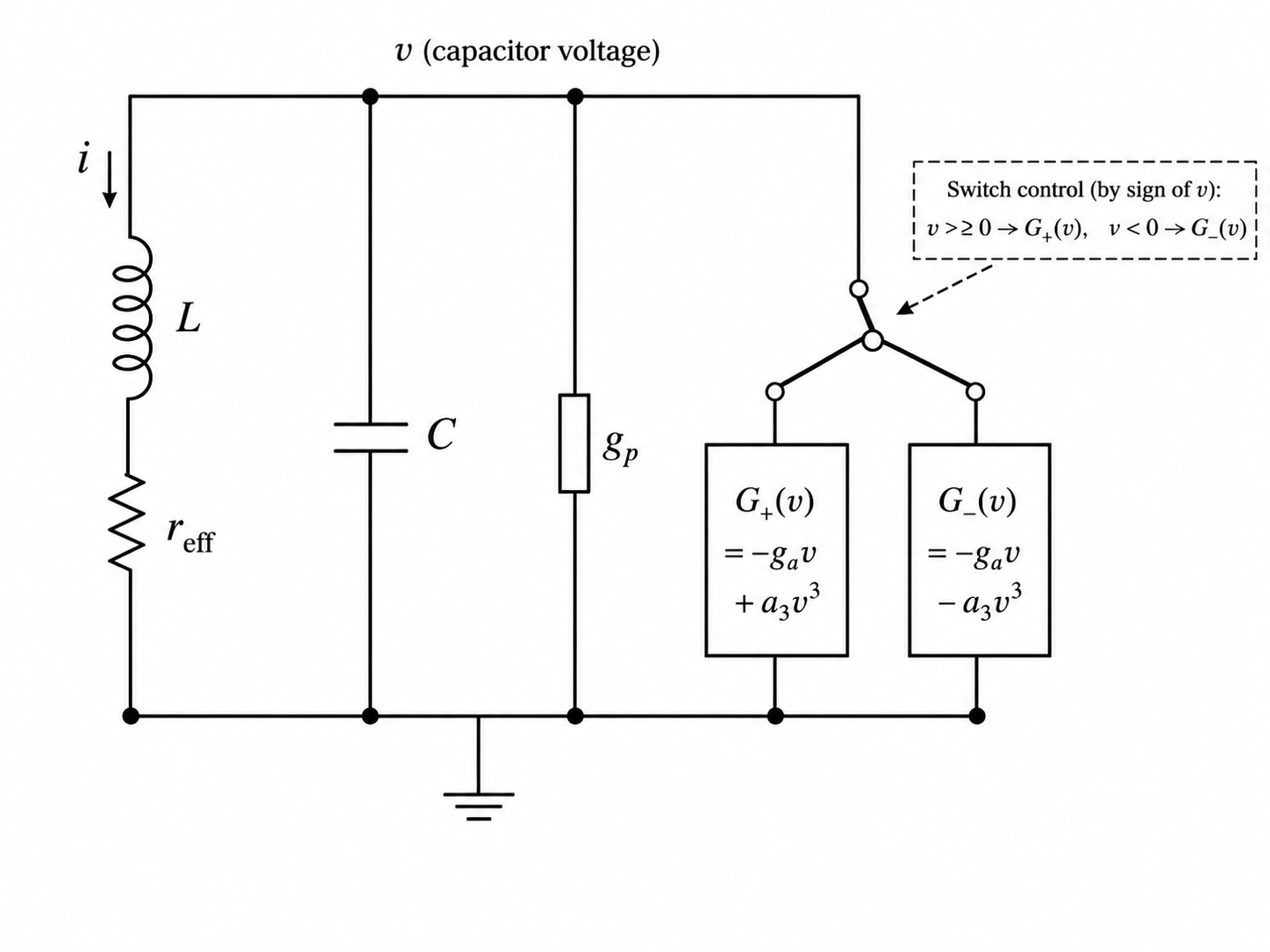}
    \caption{Modified Alpazur oscillator with switch-selected cubic
    nonlinear conductors.}
    \label{fig:modified_alpazur_circuit}
\end{figure}

The circuit model is shown in Fig.~\ref{fig:modified_alpazur_circuit}. Let $i$ and $v$ denote the inductor current and the capacitor voltage, respectively. The switch is controlled by the sign of $v$ and selects the corresponding nonlinear conductor branch. Let $g_p>0$ denote the conductance of the passive resistive branch and $g_a>0$ the magnitude of the negative linear conductance supplied by the active element. The two conductor characteristics are defined by
\begin{equation}\label{eq:modified_conductors}
    G_{+}(v)=-g_av+a_3v^3,\quad
    G_{-}(v)=-g_av-a_3v^3,
    \qquad a_3>0.
\end{equation}
Applying Kirchhoff's laws \cite{alexander2021fundamentals} gives the circuit equations
\begin{equation}\label{eq:modified_physical_alpazur}
\begin{aligned}
    L\frac{di}{dt}
    &=-v-r_{\mathrm{eff}}i,\\
    C\frac{dv}{dt}
    &=
    \begin{cases}
        i-G_{+}(v)-g_pv, & v\geq0,\\
        i-G_{-}(v)-g_pv, & v<0,
    \end{cases}
\end{aligned}
\end{equation}
where $r_{\mathrm{eff}}$ is the effective series resistance of the inductor branch. Substituting \eqref{eq:modified_conductors} into \eqref{eq:modified_physical_alpazur} yields
\begin{equation}\label{eq:modified_physical_reduced}
\begin{aligned}
    L\frac{di}{dt}
    &=-v-r_{\mathrm{eff}}i,\\
    C\frac{dv}{dt}
    &=
    \begin{cases}
        i+\mu v-a_3v^3, & v\geq0,\\
        i+\mu v+a_3v^3, & v<0,
    \end{cases}
\end{aligned}
\end{equation}
where $\mu=g_a-g_p$.

Introduce the dimensionless variables
\begin{equation*}
    X=\sqrt{L}\,i,\quad
    Y=\sqrt{C}\,v,\quad
    \tau=\frac{t}{\sqrt{LC}},
\end{equation*}
and the parameters
\begin{equation*}
    \alpha=r_{\mathrm{eff}}\sqrt{\frac{C}{L}},
    \quad
    \eta=\mu\sqrt{\frac{L}{C}},
    \quad
    \gamma=\frac{a_3\sqrt{L}}{C^{3/2}}>0.
\end{equation*}
The circuit equations \eqref{eq:modified_physical_reduced} are then transformed into
\begin{equation}\label{eq:modified_physical_reduced_2}
\begin{pmatrix}
    \dfrac{dX}{d\tau}\\[2mm]
    \dfrac{dY}{d\tau}
\end{pmatrix}
=
\begin{cases}
\begin{pmatrix}
    -\alpha X-Y\\
    X+\eta Y-\gamma Y^3
\end{pmatrix},
& Y\geq0,\\[3mm]
\begin{pmatrix}
    -\alpha X-Y\\
    X+\eta Y+\gamma Y^3
\end{pmatrix},
& Y<0.
\end{cases}
\end{equation}
The linear part has trace $\eta-\alpha$ and determinant $1-\alpha\eta$. Hence, the Hopf critical conditions are $\eta=\alpha$ and $\alpha^2<1$. Set
\begin{equation*}
    s=\sqrt{1-\alpha^2},
    \quad
    \begin{pmatrix}X\\Y\end{pmatrix}
    =
    \begin{pmatrix}
        s&-\alpha\\
        0&1
    \end{pmatrix}
    \begin{pmatrix}u\\v\end{pmatrix},
    \qquad
    \sigma=s\tau .
\end{equation*}
This transformation preserves the switching line. Finally, the scaling
$u=\sqrt{s/(3\gamma)}\,x$ and $v=\sqrt{s/(3\gamma)}\,y$ reduces the system \eqref{eq:modified_physical_reduced_2} to
\begin{equation}\label{eq:normalized_modified_alpazur}
\begin{pmatrix}
    \dfrac{dx}{d\sigma}\\[3mm]
    \dfrac{dy}{d\sigma}
\end{pmatrix}
=
\begin{cases}
\begin{pmatrix}
    -y-\dfrac{\alpha}{3\sqrt{1-\alpha^2}}y^3\\[2mm]
    x-\dfrac{1}{3}y^3
\end{pmatrix},
& y\geq0,\\[7mm]
\begin{pmatrix}
    -y+\dfrac{\alpha}{3\sqrt{1-\alpha^2}}y^3\\[2mm]
    x+\dfrac{1}{3}y^3
\end{pmatrix},
& y<0.
\end{cases}
\end{equation}

\begin{theorem}\label{thm:alpazur_weak_focus}
For each fixed $\alpha\neq0$, no small-amplitude limit cycle can bifurcate from the origin.
\end{theorem}

\begin{proof}
For $0<|\alpha|<1$, applying Algorithm~\ref{alg:algebraic_nf} gives
\[
    V_2=V_3=V_4=0,
    \quad
    V_5=
    \frac{5\pi\alpha}
    {96\sqrt{1-\alpha^2}}
    \neq0.
\]
Therefore, the origin is a weak focus.

The corresponding displacement function satisfies
\[
    \Delta(h)
    =
    V_5h^5+O(h^6)
    =
    h^5\bigl(V_5+O(h)\bigr).
\]
Since $V_5\neq0$, there exists $\delta>0$ such that $\Delta(h)\neq0$ for all $0<h<\delta$. Hence the Poincar\'e map has no nonzero fixed point sufficiently close to the origin, and consequently no small-amplitude limit cycle exists in this neighborhood.
\end{proof}

When $\alpha=0$, system~\eqref{eq:normalized_modified_alpazur} is symmetric with respect to the $x$-axis, and hence, by Lemma~\ref{lem:symmetry_center}, the origin is a center. In a practical circuit, the nonlinear conductors in the two switching branches generally have slightly different voltage--current characteristics because of component tolerances and device nonidealities. We therefore introduce small piecewise perturbations in their nonlinear coefficients to investigate the bifurcation of small-amplitude limit cycles from this center. The perturbed switching system is written as
\begin{equation}\label{eq:perturbed_modified_alpazur}
\begin{pmatrix}
    \dfrac{dx}{d\sigma}\\[3mm]
    \dfrac{dy}{d\sigma}
\end{pmatrix}
=
\begin{cases}
\begin{pmatrix}
    -y\\[2mm]
    x-\dfrac{1}{3}y^3
    +\varepsilon\left(b_{12}y^2+b_{13}y^3\right)
\end{pmatrix},
& y\geq0,\\[7mm]
\begin{pmatrix}
    -y\\[2mm]
    x+\dfrac{1}{3}y^3
    +\varepsilon\left(b_{22}y^2+b_{23}y^3\right)
\end{pmatrix},
& y<0,
\end{cases}
\end{equation}
where $0<|\varepsilon|\ll1$, $b_{12}$, $b_{13}$, $b_{22}$, $b_{23}\in \mathbb{R}$. The coefficients $b_{12}$ and $b_{22}$ describe quadratic nonlinear deviations in the two switching branches, whereas $b_{13}$ and $b_{23}$ represent small variations in their cubic voltage--current characteristics.

\begin{theorem}\label{thm:alpazur_center_cyclicity}
Consider the perturbed switching system \eqref{eq:perturbed_modified_alpazur}. At most one small-amplitude limit cycle can bifurcate from the origin, and the bound is tight. 
\end{theorem}

\begin{proof}
Applying Algorithm \ref{alg:algebraic_nf} to \eqref{eq:perturbed_modified_alpazur} and retaining the coefficient of $\varepsilon$ gives the first two Lyapunov constants
\begin{equation}\label{eq:alpazur_first_two_LC}
    V_2=\frac{4}{3}(b_{12}-b_{22}), \quad V_3=\frac{3\pi}{8}(b_{13}+b_{23}).
\end{equation}
Hence,
\[
    V_2=V_3=0
\]
if and only if
\[
    b_{12}=b_{22},
    \quad
    b_{13}=-b_{23}.
\]
Under these conditions, system~\eqref{eq:perturbed_modified_alpazur} is symmetric with respect to the $x$-axis. By Lemma~\ref{lem:symmetry_center}, the origin is a center, and all subsequent Lyapunov constants vanish. Therefore, no higher-order weak focus can occur in this perturbation family. It follows that at most one small-amplitude limit cycle can bifurcate from the origin.

It remains to show that this upper bound is attainable. Let
\[
    \delta=b_{12}-b_{22},
\]
and choose $b_{13}+b_{23}\neq0$. Then
\[
    V_2=\frac{4}{3}\delta,
    \quad
    V_3=\frac{3\pi}{8}(b_{13}+b_{23})\neq0.
\]
At $\delta=0$,
\[
    V_2(0)=0,
    \quad
    V_3(0)\neq0,
    \quad
    \frac{\partial V_2}{\partial\delta}(0)=\frac{4}{3}\neq0.
\]
Therefore, Lemma~\ref{limit cycle} implies that an arbitrarily small perturbation of $\delta$ produces exactly one small-amplitude limit cycle. Hence the cyclicity of the origin is exactly one.
\end{proof}

For a numerical illustration, we take
\[
    \varepsilon=0.1,\quad
    b_{12}=50,\quad
    b_{22}=0,\quad
    b_{13}=b_{23}=-250.
\]
Numerical integration of system~\eqref{eq:perturbed_modified_alpazur} shows the existence of a stable limit cycle, as illustrated in Fig.~\ref{fig:alpazur_limit_cycle}. 

\begin{figure}[htbp]
    \centering
    \includegraphics[width=0.45\textwidth]{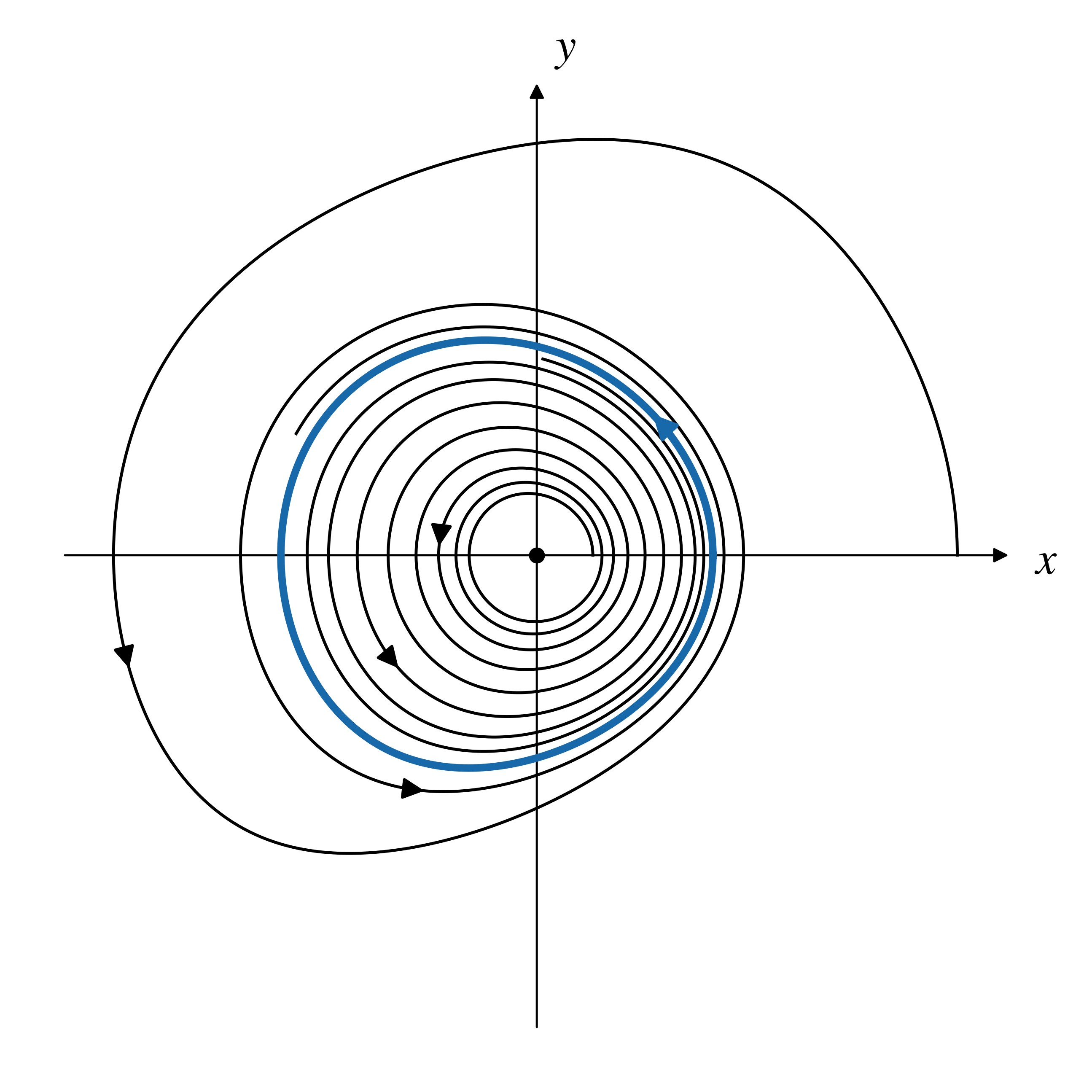}
    \caption{Phase portrait of
    system~\eqref{eq:perturbed_modified_alpazur} for
    $\varepsilon=0.1$, $b_{12}=50$, $b_{22}=0$, and
    $b_{13}=b_{23}=-250$.}
    \label{fig:alpazur_limit_cycle}
\end{figure}

From the circuit viewpoint, the stable limit cycle corresponds to a self-sustained periodic oscillation of the capacitor voltage and the inductor current. Unlike the ideal center, which possesses a continuous family of periodic orbits, the perturbed system exhibits an isolated periodic oscillation with a well-defined amplitude. This illustrates how differences in the nonlinear characteristics of the two switching branches can lead to a stable oscillatory state in the modified Alpazur circuit.

\section{Conclusion}\label{Conclusion}
We have developed a symbolic normal-form procedure for computing Lyapunov constants in planar switching systems. After the substitution $z=e^{\mathrm{i}\theta}$, the angular dependence of the radial equations is represented in an extended Laurent algebra. The half-period mean operator determines the normal-form coefficients, while the algebraic integration operator gives the normalized solutions of the homological equations. This formulation replaces repeated symbolic integration by finite coefficient operations. The algorithm was applied to a
generalized quartic switching Li{\'e}nard system and to a circuit-level modification of the Alpazur oscillator. In the first example, the recursive computation yields center conditions and a parameter configuration associated with ten small-amplitude limit cycles. In the second example, the algorithm determines the center conditions for small nonlinear perturbations of the two switching branches and shows that the first-order cyclicity of the center is one. Consequently, at most one small-amplitude limit cycle can bifurcate from the center within this perturbation family, and this upper bound is attainable. This example further illustrates the applicability of the proposed symbolic procedure to the analysis of nontrivial oscillatory behavior in physically motivated switching systems.

\appendix
\section{The algebraic expression in Example 1} \label{appendix_1}
The parameter expressions omitted from Section \ref{example:1} are provided below.
\begin{align*}
M_{7} &= \pi \Big( 875232 a_{21}^6 + 14542448 a_{21}^5 a_{22} + 43181792 a_{21}^4 a_{22}^2 + 44563008 a_{21}^3 a_{22}^3 \\
      &\quad + 13374272 a_{21}^2 a_{22}^4 - 1674160 a_{21} a_{22}^5 + 12866592 a_{21}^4 b_{32} + 6118416 a_{21}^4 b_{41} \\
      &\quad + 10332504 a_{21}^3 a_{22} b_{32} - 5214672 a_{21}^3 a_{22} b_{41} - 20835288 a_{21}^2 a_{22}^2 b_{32} - 28784592 a_{21}^2 a_{22}^2 b_{41} \\
      &\quad - 18301200 a_{21} a_{22}^3 b_{32} - 17451504 a_{21} a_{22}^3 b_{41} - 71244000 a_{21}^4 - 119354400 a_{21}^3 a_{22} \\
      &\quad - 24976800 a_{21}^2 a_{22}^2 + 7252470 a_{21}^2 b_{32}^2 - 10044324 a_{21}^2 b_{32} b_{41} + 3538080 a_{21}^2 b_{41}^2 \\
      &\quad + 23133600 a_{21} a_{22}^3 + 8434935 a_{21} a_{22} b_{32}^2 + 4244076 a_{21} a_{22} b_{32} b_{41} + 7076160 a_{21} a_{22} b_{41}^2 \\
      &\quad + 14288400 a_{22}^2 b_{32} b_{41} + 3538080 a_{22}^2 b_{41}^2 + 23092200 a_{21}^2 b_{32} - 3931200 a_{21}^2 b_{41} \\
      &\quad + 7216200 a_{21} a_{22} b_{32} - 7862400 a_{21} a_{22} b_{41} - 15876000 a_{22}^2 b_{32} - 3931200 a_{22}^2 b_{41} \\
      &\quad - 3231900 b_{32}^2 b_{41} + 3591000 b_{32}^2 \Big).
\end{align*}

\begin{align*}
M_8 &= 1302673428000 \pi a_{21}^6 + 21644616042000 \pi a_{21}^5 a_{22} + 64270699668000 \pi a_{21}^4 a_{22}^2 \\
    &\quad + 66326467032000 \pi a_{21}^3 a_{22}^3 + 19905932088000 \pi a_{21}^2 a_{22}^4 - 2491777890000 \pi a_{21} a_{22}^5 \\
    &\quad + 8739923216384 a_{21}^7 + 21890113658880 a_{21}^6 a_{22} - 15524481490944 a_{21}^5 a_{22}^2 \\
    &\quad - 107342461325312 a_{21}^4 a_{22}^3 - 119082307233792 a_{21}^3 a_{22}^4 - 35246185451520 a_{21}^2 a_{22}^5 \\
    &\quad + 5168332390400 a_{21} a_{22}^6 + 19150313868000 \pi a_{21}^4 b_{32} + 9106497414000 \pi a_{21}^4 b_{41} \\
    &\quad + 15378640641000 \pi a_{21}^3 a_{22} b_{32} - 7761387438000 \pi a_{21}^3 a_{22} b_{41} - 31010721777000 \pi a_{21}^2 a_{22}^2 b_{32} \\
    &\quad - 42842267118000 \pi a_{21}^2 a_{22}^2 b_{41} - 27239048550000 \pi a_{21} a_{22}^3 b_{32}  \\
    &\quad + 26755230592512 a_{21}^5 b_{32} + 13032430571520 a_{21}^5 b_{41} + 33593961959424 a_{21}^4 a_{22} b_{32} \\
    &\quad + 14013997166592 a_{21}^4 a_{22} b_{41} + 14845738102272 a_{21}^3 a_{22}^2 b_{32} + 30558375788544 a_{21}^3 a_{22}^2 b_{41} \\
    &\quad + 46478020930560 a_{21}^2 a_{22}^3 b_{32} + 71204482363392 a_{21}^2 a_{22}^3 b_{41} + 38471014195200 a_{21} a_{22}^4 b_{32} \\
    &\quad + 41627673169920 a_{21} a_{22}^4 b_{41} - 106037788500000 \pi a_{21}^4 - 177644105100000 \pi a_{21}^3 a_{22} \\
    &\quad - 37174844700000 \pi a_{21}^2 a_{22}^2 + 10794395036250 \pi a_{21}^2 b_{32}^2 - 14949720733500 \pi a_{21}^2 b_{32} b_{41} \\
    &\quad + 5265989820000 \pi a_{21}^2 b_{41}^2 + 34431471900000 \pi a_{21} a_{22}^3 + 12554346380625 \pi a_{21} a_{22} b_{32}^2 \\
    &\quad + 6316776616500 \pi a_{21} a_{22} b_{32} b_{41} + 10531979640000 \pi a_{21} a_{22} b_{41}^2 \\
    &\quad + 5265989820000 \pi a_{22}^2 b_{41}^2 - 84635927347200 a_{21}^5 - 91674005913600 a_{21}^4 a_{22} \\
    &\quad + 24313203302400 a_{21}^3 a_{22}^2 + 13869935493120 a_{21}^3 b_{32}^2 - 7744845279744 a_{21}^3 b_{32} b_{41} \\
    &\quad + 8719337539584 a_{21}^3 b_{41}^2 - 14895285043200 a_{21}^2 a_{22}^3 + 18711830165760 a_{21}^2 a_{22} b_{32}^2 \\
    &\quad + 20979488871936 a_{21}^2 a_{22} b_{32} b_{41} + 8403833991168 a_{21}^2 a_{22} b_{41}^2 - 46246566912000 a_{21} a_{22}^4 \\
    &\quad + 3010668134400 a_{21} a_{22}^2 b_{32}^2 - 2271714232320 a_{21} a_{22}^2 b_{32} b_{41} - 9350344636416 a_{21} a_{22}^2 b_{41}^2 \\
    &\quad - 30996048384000 a_{22}^3 b_{32} b_{41} - 9034841088000 a_{22}^3 b_{41}^2 + 34369853175000 \pi a_{21}^2 b_{32} \\
    &\quad - 5851099800000 \pi a_{21}^2 b_{41} + 10740411675000 \pi a_{21} a_{22} b_{32} - 11702199600000 \pi a_{21} a_{22} b_{41} \\
    &\quad - 23629441500000 \pi a_{22}^2 b_{32} - 5851099800000 \pi a_{22}^2 b_{41} - 4810279162500 \pi b_{32}^2 b_{41} \\
    &\quad + 33859687680000 a_{21}^3 b_{32} + 1298289254400 a_{21}^3 b_{41} + 12185641728000 a_{21}^2 a_{22} b_{32} \\
    &\quad + 12635290828800 a_{21}^2 a_{22} b_{41} + 12766007808000 a_{21} a_{22}^2 b_{32} + 21375713894400 a_{21} a_{22}^2 b_{41} \\
    &\quad - 5882967014400 a_{21} b_{32}^2 b_{41} - 5711702169600 a_{21} b_{32} b_{41}^2 + 34440053760000 a_{22}^3 b_{32} \\
    &\quad + 10038712320000 a_{22}^3 b_{41} - 877879296000 a_{22} b_{32}^2 b_{41} - 5711702169600 a_{22} b_{32} b_{41}^2 \\
    &\quad + 5344754625000 \pi b_{32}^2 + 8941363200000 a_{21}^3 + 17882726400000 a_{21}^2 a_{22}  \\
    &\quad + 8941363200000 a_{21} a_{22}^2 + 6536630016000 a_{21} b_{32}^2 + 3663926784000 a_{21} b_{32} b_{41} \\
    &\quad  + 975421440000 a_{22} b_{32}^2 + 3663926784000 a_{22} b_{32} b_{41} - 6706022400000 a_{21} b_{32} \\
    &\quad - 6706022400000 a_{22} b_{32}  + 21266497350000 \pi a_{22}^2 b_{32} b_{41} - 25974382266000 \pi a_{21} a_{22}^3 b_{41}.
\end{align*}

\begin{align*}
b_{41} &= \frac{N}{108 \cdot D}, \quad \text{where} \\
K &= \frac{20475}{22208} a_{21} \pi + \frac{20475}{22208} a_{22} \pi + \frac{834311}{546525} a_{21}^2 - \frac{29}{525} a_{22} a_{21} - \frac{4940}{3123} a_{22}^2, \\
N &= 6685837461184 a_{21}^8 + 9081072000000 a_{21}^4 + 92360594940000 a_{21}^6 \\
  &\quad - 3595928137416 a_{21}^6 K + 17029949319090 a_{21}^4 K^2 - 6196226400000 a_{22}^4 K \\
  &\quad - 147409646810080 a_{21}^4 a_{22}^4 - 35764136132016 a_{21}^3 a_{22}^5 + 3986731268080 a_{21}^2 a_{22}^6 \\
  &\quad + 907153624000 a_{21} a_{22}^7 - 98029056945360 a_{21}^6 a_{22}^2 - 194698277856160 a_{21}^5 a_{22}^3 \\
  &\quad - 5210874662000 a_{21}^7 a_{22} - 124779383440800 a_{21}^2 a_{22}^4 - 201478942142400 a_{21}^4 a_{22}^2 \\
  &\quad - 360144487862400 a_{21}^3 a_{22}^3 + 13027822080000 a_{21} a_{22}^5 + 113218935139200 a_{21}^5 a_{22} \\
  &\quad + 9081072000000 a_{21} a_{22}^3 + 27243216000000 a_{21}^3 a_{22} + 27243216000000 a_{21}^2 a_{22}^2 \\
  &\quad + 11890968500250 a_{21}^2 K^3 - 6810804000000 a_{22}^2 K + 13829708552625 a_{21} a_{22} K^3 \\
  &\quad - 164895976537200 a_{21}^3 a_{22} K + 29576528279235 a_{21}^3 a_{22} K^2 + 7474102716645 a_{21}^2 a_{22}^2 K^2 \\
  &\quad + 60439306268592 a_{21}^5 a_{22} K - 5072476243500 a_{21} a_{22}^3 K^2 + 103150301508000 a_{21} a_{22}^3 K \\
  &\quad - 11136949290000 a_{21} a_{22}^5 K + 206063373624624 a_{21}^3 a_{22}^3 K - 13621608000000 a_{21} a_{22} K \\
  &\quad + 19786154472000 a_{21} a_{22} K^2 + 52510681605600 a_{21}^2 a_{22}^4 K + 84669916793400 a_{21}^2 a_{22}^2 K \\
  &\quad + 206450977135032 a_{21}^4 a_{22}^2 K + 35512126272000 a_{21}^2 K^2 + 5887713825000 K^3 \\
  &\quad - 140219365422600 a_{21}^4 K - 6810804000000 a_{21}^2 K - 15725971800000 a_{22}^2 K^2, \\
D &= 19240027916 a_{21}^6 - 380325272096 a_{21}^5 a_{22} - 1228801446184 a_{21}^4 a_{22}^2 \\
  &\quad - 1212052426816 a_{21}^3 a_{22}^3 - 355201743044 a_{21}^2 a_{22}^4 + 27614537600 a_{21} a_{22}^5 \\
  &\quad - 252833457669 a_{21}^4 K + 61488269562 a_{21}^3 a_{22} K + 829841692131 a_{21}^2 a_{22}^2 K \\
  &\quad + 463884744900 a_{21} a_{22}^3 K - 51635220000 a_{22}^4 K - 103315789200 a_{21}^4 \\
  &\quad - 309947367600 a_{21}^3 a_{22} - 309947367600 a_{21}^2 a_{22}^2 + 132908164200 a_{21}^2 K^2 \\
  &\quad - 103315789200 a_{21} a_{22}^3 + 1858399200 a_{21} a_{22} K^2 - 131049765000 a_{22}^2 K^2 \\
  &\quad + 25225200000 a_{21}^2 K + 50450400000 a_{21} a_{22} K + 25225200000 a_{22}^2 K \\
  &\quad + 49064281875 K^3.
\end{align*}




\bibliographystyle{elsarticle-num} 

\bibliography{references}

@String{Academic = "Academic Press" }

@String{Springer = "Springer-Verlag" }

@book{guckenheimer2013,
  author    = {Guckenheimer, J. and Holmes, P.},
  title     = {Nonlinear Oscillations, Dynamical Systems, and Bifurcations of Vector Fields},
  series    = {Applied Mathematical Sciences},
  volume    = {42},
  publisher = {Springer},
  year      = {1983}
}

@book{marsden2012hopf,
  author    = {Marsden, J. E. and McCracken, M.},
  title     = {The {Hopf} Bifurcation and Its Applications},
  series    = {Applied Mathematical Sciences},
  volume    = {19},
  publisher = {Springer},
  year      = {1976}
}

@book{han2012normal,
  author    = {Han, M. and Yu, P.},
  title     = {Normal Forms, {Melnikov} Functions and Bifurcations of Limit Cycles},
  series    = {Applied Mathematical Sciences},
  volume    = {181},
  publisher = {Springer},
  year      = {2012}
}

@book{perko2013differential,
  author    = {Perko, L.},
  title     = {Differential Equations and Dynamical Systems},
  edition   = {3},
  series    = {Texts in Applied Mathematics},
  volume    = {7},
  publisher = {Springer},
  year      = {2001}
}

@article{li2015center,
  author  = {Li, F. and Yu, P. and Tian, Y. and Liu, Y.},
  title   = {Center and Isochronous Center Conditions for Switching Systems Associated with Elementary Singular Points},
  journal = {Communications in Nonlinear Science and Numerical Simulation},
  volume  = {28},
  number  = {1--3},
  pages   = {81--97},
  year    = {2015}
}

@book{bernardo2008piecewise,
  author    = {di Bernardo, M. and Budd, C. J. and Champneys, A. R. and Kowalczyk, P.},
  title     = {Piecewise-Smooth Dynamical Systems: Theory and Applications},
  series    = {Applied Mathematical Sciences},
  volume    = {163},
  publisher = {Springer},
  year      = {2008}
}

@article{yu2021eighteen,
  author  = {Yu, P. and Han, M. and Zhang, X.},
  title   = {Eighteen Limit Cycles Around Two Symmetric Foci in a Cubic Planar Switching Polynomial System},
  journal = {Journal of Differential Equations},
  volume  = {275},
  pages   = {939--959},
  year    = {2021}
}

@article{buzzi2013piecewise,
  author  = {Buzzi, C. and Pessoa, C. and Torregrosa, J.},
  title   = {Piecewise Linear Perturbations of a Linear Center},
  journal = {Discrete and Continuous Dynamical Systems},
  volume  = {33},
  number  = {9},
  pages   = {3915--3936},
  year    = {2013}
}

@article{chen2024nilpotent,
  author  = {Chen, T. and Li, F. and Yu, P.},
  title   = {Nilpotent Center Conditions in Cubic Switching Polynomial {Li{\'e}nard} Systems by Higher-Order Analysis},
  journal = {Journal of Differential Equations},
  volume  = {379},
  pages   = {258--289},
  year    = {2024}
}

@article{liang2016number,
  author  = {Liang, F. and Han, M.},
  title   = {On the Number of Limit Cycles in Small Perturbations of a Piecewise Linear Hamiltonian System with a Heteroclinic Loop},
  journal = {Chinese Annals of Mathematics, Series B},
  volume  = {37},
  number  = {2},
  pages   = {267--280},
  year    = {2016}
}

@book{filippov1988differential,
  author    = {Filippov, A. F.},
  title     = {Differential Equations with Discontinuous Righthand Sides: Control Systems},
  series    = {Mathematics and Its Applications},
  volume    = {18},
  publisher = {Kluwer Academic Publishers},
  year      = {1988}
}

@article{freire1998bifurcation,
  author  = {Freire, E. and Ponce, E. and Rodrigo, F. and Torres, F.},
  title   = {Bifurcation Sets of Continuous Piecewise Linear Systems with Two Zones},
  journal = {International Journal of Bifurcation and Chaos},
  volume  = {8},
  number  = {11},
  pages   = {2073--2097},
  year    = {1998}
}

@article{llibre2004existence,
  author  = {Llibre, J. and Teruel, A. E.},
  title   = {Existence of {Poincar{\'e}} Maps in Piecewise Linear Differential Systems in {$\mathbb{R}^n$}},
  journal = {International Journal of Bifurcation and Chaos},
  volume  = {14},
  number  = {8},
  pages   = {2843--2851},
  year    = {2004}
}

@article{guo2019bifurcation,
  author  = {Guo, L. and Yu, P. and Chen, Y.},
  title   = {Bifurcation Analysis on a Class of Three-Dimensional Quadratic Systems with Twelve Limit Cycles},
  journal = {Applied Mathematics and Computation},
  volume  = {363},
  pages   = {124577},
  year    = {2019}
}

@article{chen2026global,
  author  = {Chen, T. and Peng, J.},
  title   = {Global Center and Limit Cycles in Generalized Piecewise Cubic {Li{\'e}nard} Systems},
  journal = {Nonlinear Analysis: Real World Applications},
  volume  = {87},
  pages   = {104452},
  year    = {2026}
}

@article{kousaka1999bifurcation,
  author  = {Kousaka, T. and Ueta, T. and Kawakami, H.},
  title   = {Bifurcation of Switched Nonlinear Dynamical Systems},
  journal = {IEEE Transactions on Circuits and Systems II: Analog and Digital Signal Processing},
  volume  = {46},
  number  = {7},
  pages   = {878--885},
  year    = {1999}
}

@article{gasull2003center,
  author  = {Gasull, A. and Torregrosa, J.},
  title   = {Center-Focus Problem for Discontinuous Planar Differential Equations},
  journal = {International Journal of Bifurcation and Chaos},
  volume  = {13},
  number  = {7},
  pages   = {1755--1765},
  year    = {2003}
}

@article{smale1998mathematical,
  author  = {Smale, S.},
  title   = {Mathematical Problems for the Next Century},
  journal = {The Mathematical Intelligencer},
  volume  = {20},
  number  = {2},
  pages   = {7--15},
  year    = {1998}
}

@article{tian2011hopf,
  author  = {Tian, Y. and Han, M.},
  title   = {{Hopf} Bifurcation for Two Types of {Li{\'e}nard} Systems},
  journal = {Journal of Differential Equations},
  volume  = {251},
  number  = {4--5},
  pages   = {834--859},
  year    = {2011}
}

@article{llibre2015limit,
  author  = {Llibre, J. and Teixeira, M. A.},
  title   = {Limit Cycles for {$m$}-Piecewise Discontinuous Polynomial {Li{\'e}nard} Differential Equations},
  journal = {Zeitschrift f{\"u}r Angewandte Mathematik und Physik},
  volume  = {66},
  number  = {1},
  pages   = {51--66},
  year    = {2015}
}

@incollection{lins1977lienard,
  author    = {Lins, A. and de Melo, W. and Pugh, C. C.},
  title     = {On {Li{\'e}nard}'s Equation},
  booktitle = {Geometry and Topology},
  series    = {Lecture Notes in Mathematics},
  volume    = {597},
  pages     = {335--357},
  publisher = {Springer},
  year      = {1977}
}

@article{tian2015center,
  author  = {Tian, Y. and Yu, P.},
  title   = {Center Conditions in a Switching {Bautin} System},
  journal = {Journal of Differential Equations},
  volume  = {259},
  number  = {3},
  pages   = {1203--1226},
  year    = {2015}
}

@article{colombo2009two,
  author  = {Colombo, A. and Lamiani, P. and Benadero, L. and di Bernardo, M.},
  title   = {Two-Parameter Bifurcation Analysis of the Buck Converter},
  journal = {SIAM Journal on Applied Dynamical Systems},
  volume  = {8},
  number  = {4},
  pages   = {1507--1522},
  year    = {2009}
}

@book{romanovski2009center,
  author    = {Romanovski, V. G. and Shafer, D. S.},
  title     = {The Center and Cyclicity Problems: A Computational Algebra Approach},
  publisher = {Birkh{\"a}user},
  year      = {2009}
}

@article{christopher1999small,
  author  = {Christopher, C. and Lynch, S.},
  title   = {Small-Amplitude Limit Cycle Bifurcations for {Li{\'e}nard} Systems with Quadratic or Cubic Damping or Restoring Forces},
  journal = {Nonlinearity},
  volume  = {12},
  number  = {4},
  pages   = {1099--1112},
  year    = {1999}
}

@book{alexander2021fundamentals,
  author    = {Alexander, C. K. and Sadiku, M. N. O.},
  title     = {Fundamentals of Electric Circuits},
  edition   = {7},
  publisher = {McGraw-Hill Education},
  year      = {2021}
}
\end{document}